\documentclass{amsart}
\usepackage{graphicx, color}
\usepackage{amscd}
\usepackage{amsmath,empheq}
\usepackage{amsfonts}
\usepackage{amssymb}
\usepackage{mathrsfs}

\usepackage[all]{xy}

\newtheorem{theorem}{Theorem}

\newtheorem{lemma}[theorem]{Lemma}
\newtheorem{prop}[theorem]{Proposition}
\newtheorem{remark}{Remark}

\newtheorem{claim}{Claim}

\newenvironment{proof-sketch}{\noindent{\bf Sketch of Proof}\hspace*{1em}}{\qed\bigskip}

\everymath{\displaystyle}

\newcommand{\RR}{\mathbb R}
\newcommand{\NN}{\mathbb N}

\newcommand{\ZZ}{\mathbb Z}

\renewcommand{\leq}{\leqslant}

\renewcommand{\geq}{\geqslant}

\begin{document}

\title[Existence and multiplicity of solutions for resonant $(p,2)$-equations]{Existence and multiplicity of solutions for resonant $(p,2)$-equations}

\author[N.S. Papageorgiou]{Nikolaos S. Papageorgiou}
\address[N.S. Papageorgiou]{Department of Mathematics, National Technical University,
				Zografou Campus, Athens 15780, Greece}
\email{\tt npapg@math.ntua.gr}

\author[V.D. R\u{a}dulescu]{Vicen\c{t}iu D. R\u{a}dulescu}
\address[V.D. R\u{a}dulescu]{Department of Mathematics, Faculty of Sciences, King
Abdulaziz University, P.O. Box 80203, Jeddah 21589, Saudi Arabia \& Institute of Mathematics ``Simion Stoilow" of the Romanian Academy, P.O. Box 1-764, 014700 Bucharest, Romania}
\email{\tt vicentiu.radulescu@imar.ro}

\author[D.D. Repov\v s]{Du\v san D. Repov\v s}
\address[D.D. Repov\v s]{Faculty of Education and Faculty of Mathematics and Physics,
University of Ljubljana, SI-1000 Ljubljana, Slovenia}
\email{dusan.repovs@guest.arnes.si}


\begin{abstract}
We consider Dirichlet elliptic equations driven by the sum of a $p$-Laplacian $(2<p)$ and a Laplacian. The conditions on the reaction term imply that the problem is resonant at both $\pm\infty$ and at zero. We prove an existence theorem (producing one nontrivial smooth solution) and a multiplicity theorem (producing five nontrivial smooth solutions, four of constant sign and the fifth nodal; the solutions are ordered). Our approach uses variational methods and critical groups.
\end{abstract}

\maketitle


\section{Introduction}

Let $\Omega\subseteq\RR^N$ be a bounded domain with a $C^2$-boundary $\partial\Omega$. In this paper we study the following nonlinear, nonhomogeneous Dirichlet problem:
\begin{equation}\label{eq1}
	\left\{\begin{array}{ll}
	-\Delta_p u(z)-\Delta u(z)=f(z,u(z))&\mbox{in}\ \Omega,\\
	u|_{\partial\Omega}=0,\quad 2<p.&
	\end{array}\right\}
\end{equation}

Here, for $r\in(1,\infty)$, we denote by $\Delta_r$ the $r$-Laplacian defined by
$$\Delta_ru={\rm div}\,(|Du|^{r-2}Du)\ \mbox{for all}\ u\in W^{1,r}_0(\Omega).$$

When $r=2$, we write $\Delta_2=\Delta$ (the standard Laplace differential operator). The reaction term $f(z,x)$ is a Carath\'eodory function (that is, for all $x\in\RR$, $z\mapsto f(z,x)$ is measurable and for almost all $z\in\Omega$, $x\mapsto f(z,x)$ is continuous). We assume that for almost all $z\in\Omega$, $f(z,\cdot)$ is $(p-1)$-sublinear near $\pm\infty$ and asymptotically as $x\rightarrow\pm\infty$ the quotient $\frac{f(z,x)}{|x|^{p-2}x}$ interacts with the variational part of the spectrum of $(-\Delta_p,W^{1,p}_{0}(\Omega))$ (resonant problem). Equations driven by the sum of a $p$-Laplacian and a Laplacian (known as $(p,2)$-equations) have recently been studied by Aizicovici, Papageorgiou and Staicu \cite{3}, Cingolani and Degiovanni \cite{10}, Papageorgiou and R\u adulescu \cite{23,24},  Papageorgiou, R\u adulescu and Repov\v{s} \cite{27}, Papageorgiou and Winkert \cite{28}, Sun \cite{31}, Sun, Zhang and Su \cite{32}. The aforementioned works, either do not consider
  resonant at $\pm\infty$ equations (see Aizicovici, Papageorgiou and Staicu \cite{3}, Cingolani and Degiovanni \cite{10}, Sun \cite{31}, Sun, Zhang and Su \cite{32}) or the resonance is with respect to the principal eigenvalue (see Papageorgiou and R\u adulescu \cite{23,24}, Papageorgiou, R\u adulescu and Repov\v{s} \cite{27}, Papageorgiou and Winkert \cite{28}). For $p\neq 2$, we not have a complete knowledge of the spectrum of $(-\Delta_p,W^{1,p}_{0}(\Omega))$, the eigenspaces are not linear subspaces of $W^{1,p}_{0}(\Omega)$ and the Sobolev space $W^{1,p}_{0}(\Omega)$ cannot be expressed as a direct sum of the eigenspaces. All these negative facts make difficult the study of  problems with resonance at higher parts of the spectrum. Our present paper is closer to those of Cingolani and Degiovanni \cite{10}
and of Papageorgiou and R\u adulescu \cite{23}. Compared to Cingolani and Degiovanni \cite{10},
we allow for resonance to occur and so we improve their existence theorem. Compared with the work of Papageorgiou and R\u adulescu \cite{23}, the resonance is with respect to any variational eigenvalue of $(-\Delta_p,W^{1,p}_{0}(\Omega))$, not only the principal one.

Using tools from Morse theory and variational methods based on the critical point theory, we prove existence and multiplicity theorems for resonant $(p,2)$-equations. We mention that $(p,2)$-equations arise in problems of mathematical physics. The Dirichlet $(p,2)$-problem treated in this paper  models some phenomena in quantum physics as first
pointed out by Benci, Fortunato and Pisani \cite{benci}. We refer to the works of Benci, D'Avenia, Fortunato and Pisani \cite{5} (in quantum physics) and Cherfils and Ilyasov \cite{9} (in plasma physics). Related results on $(p,q)$-Laplacian problems are due to Marano, Mosconi and Papageorgiou \cite{marano} and Mugnai and Papageorgiou \cite{mugnai}.

In the next section we briefly recall the main mathematical tools which will be used in the sequel.

\section{Mathematical Background}

Let $X$ be a Banach space and $X^*$ its topological dual. By $\left\langle \cdot,\cdot\right\rangle$ we denote the duality brackets for the dual pair $(X^*,X)$. Also, let $\varphi\in C^1(X,\RR)$. We say that $\varphi$ satisfies the ``Cerami condition'' (the ``C-condition'' for short), if the following property holds:
\begin{center}
``Every sequence $\{u_n\}_{n\geq 1}\subseteq X$ such that $\{\varphi(u_n)\}_{n\geq 1}\subseteq\RR$ is bounded and
$$(1+||u_n||)\varphi'(u_n)\rightarrow 0\ \mbox{in}\ X^*\ \mbox{as}\ n\rightarrow\infty,$$
admits a strongly convergent subsequence''.
\end{center}

This compactness-type condition on the functional $\varphi$ leads to a deformation theorem from which one derives the minimax theory of the critical values of $\varphi$. A basic result in this theory is the celebrated ``mountain pass theorem'' due to Ambrosetti and Rabinowitz \cite{4}. Hence we state the result in a slightly more general form (see, for example, Gasinski and Papageorgiou \cite[p. 648]{15}).
\begin{theorem}\label{th1}
	Let $X$ be a Banach space and assume that $\varphi\in C^1(X,\RR)$ satisfies the C-condition, $u_0,\, u_1\in X$, $||u_1-u_0||>\rho>0$,
	$$\max\{\varphi(u_0),\varphi(u_1)\}<\inf[\varphi(u):||u-u_0||=\rho]=m_{\rho}$$
	and $c=\inf\limits_{\gamma\in\Gamma}\max\limits_{0\leq t\leq 1}\varphi(\gamma(t))$, where $\Gamma=\{\gamma\in C([0,1],X):\gamma(0)=u_0,\gamma(1)=u_1\}$.

Then $c\geq m_{\rho}$ and $c$ is a critical value of $\varphi$ (that is, there exists $u\in X$ such that $\varphi'(u)=0$, $\varphi(u)=c$).
\end{theorem}

Three Banach spaces will be central in our analysis of problem (\ref{eq1}). We refer to the Dirichlet Sobolev spaces $W^{1,p}_{0}(\Omega)$ and $H^1_0(\Omega)$ and the Banach space $C^1_0(\overline{\Omega})=\{u\in C^1(\overline{\Omega}):u|_{\partial\Omega}=0\}$.

By Poincar\'e's inequality, the norm of $W^{1,p}_{0}(\Omega)$ can be defined by
$$||u||=||Du||_p\ \mbox{for all}\ u\in W^{1,p}_{0}(\Omega).$$

The space $H^1_0(\Omega)$ is a Hilbert space and again the Poincar\'e inequality implies that we can choose as inner product
$$(u,h)=(Du,Dh)_{L^2(\Omega,\RR^N)}\ \mbox{for all}\ u,h\in H^1_0(\Omega).$$

The corresponding norm is
$$||u||_{H^1_0(\Omega)}=||Du||_2\ \mbox{for all}\ u\in H^1_0(\Omega).$$

The Banach space $C^1_0(\overline{\Omega})$ is an ordered Banach space with positive cone
$$C_+=\{u\in C^1_0(\overline{\Omega}):u(z)\geq 0\ \mbox{for all}\ z\in\overline{\Omega}\}.$$

This cone has a nonempty interior, given by
$${\rm int}\, C_+=\{u\in C_+:u(z)>0\ \mbox{for all}\ z\in\Omega,\ \left.\frac{\partial u}{\partial n}\right|_{\partial\Omega}<0\}.$$

Here, $\frac{\partial u}{\partial n}$ is the usual normal derivative defined by $\frac{\partial u}{\partial n}=(Du,n)_{\RR^N}$, with $n(\cdot)$ being the outward unit normal on $\partial\Omega$. Recall that $C^1_0(\overline{\Omega})$ is dense in both $W^{1,p}_{0}(\Omega)$ and $H^1_0(\Omega)$.

Given $x\in\RR$, we set $x^{\pm}=\max\{\pm x,0\}$ and then define $u^{\pm}(\cdot)=u(\cdot)^{\pm}$ for all $u\in W^{1,p}_{0}(\Omega)$. We know that
$$u^{\pm}\in W^{1,p}_{0}(\Omega),\ u=u^+-u^-,\ |u|=u^++u^-.$$

Also,  we denote the Lebesgue measure on $\RR^N$ by $|\cdot|_N$ and if $g:\Omega\times\RR\rightarrow\RR$ is a measurable function (for example, a Carath\'eodory function), we define the Nemytskii map corresponding to $g(\cdot,\cdot)$ by
$$N_g(u)(\cdot)=g(\cdot,u(\cdot))\ \mbox{for all}\ u\in W^{1,p}_{0}(\Omega).$$

We will use the spectra of the operators $(-\Delta_p,W^{1,p}_{0}(\Omega))$ and $(-\Delta,H^1_0(\Omega))$. We start with the spectrum of $(-\Delta_p,W^{1,p}_{0}(\Omega))$. So, consider the following nonlinear eigenvalue problem:
\begin{equation}\label{eq2}
	-\Delta_p u(z)=\hat{\lambda}|u(z)|^{p-2}u(z)\ \mbox{in}\ \Omega,\quad u|_{\partial\Omega}=0\ (1<p<\infty).
\end{equation}

We say that $\hat{\lambda}\in\RR$ is an ``eigenvalue'' of ($-\Delta_p,W^{1,p}_{0}(\Omega)$), if problem (\ref{eq2}) admits a nontrivial solution $\hat{u}\in W^{1,p}_{0}(\Omega)$, known as an ``eigenfunction'' corresponding to $\hat{\lambda}$. We know that there exists the smallest eigenvalue $\hat{\lambda}_1(p)>0$, which has the following properties:
\begin{itemize}
	\item $\hat{\lambda}_1(p)$ is isolated in the spectrum $\hat{\sigma}(p)$ of $(-\Delta_p,W^{1,p}_{0}(\Omega))$ (that is, there exists $\epsilon>0$ such that $(\hat{\lambda}_1(p),\hat{\lambda}_1(p)+\epsilon)\cap\hat{\sigma}(p)=\emptyset$);
	\item $\hat{\lambda}_1(p)$ is simple (that is, if $\hat{u},\tilde{u}\in W^{1,p}_{0}(\Omega)$ are eigenfunctions corresponding to $\hat{\lambda}_1(p)$, then $\hat{u}=\xi\tilde{u}$ with $\xi\in\RR\backslash\{0\}$);
 \begin{equation}\label{eq3}
		\bullet\hat{\lambda}_1(p)=\inf\left[\frac{||Du||^p_p}{||u||^p_p}:u\in W^{1,p}_{0}(\Omega),u\neq 0\right].
	\end{equation}
\end{itemize}

In (\ref{eq3}) the infimum is realized on the one-dimensional eigenspace corresponding to $\hat{\lambda}_1(p)$. The above properties imply that the elements of this eigenspace do not change sign. We point out that the nonlinear regularity theory (see, for example, Gasinski and Papageorgiou \cite[p. 737]{15}), implies that all eigenfunctions of $(-\Delta_p,W^{1,p}_{0}(\Omega))$ belong to $C^1_0(\overline{\Omega})$. By $\hat{u}_1(p)$ we denote the positive $L^p$-normalized (that is, $||\hat{u}_1(p)||_p=1$) eigenfunction corresponding to $\hat{\lambda}_1(p)>0$. As we have already mentioned, $\hat{u}_1(p)\in C_+\backslash\{0\}$ and in fact, the nonlinear maximum principle (see for example Gasinski and Papageorgiou \cite[p. 738]{15}) implies that $\hat{u}_1(p)\in {\rm int}\,C_+$. An eigenfunction $\hat{u}$ which corresponds to an eigenvalue $\hat{\lambda}\neq\hat{\lambda}_1(p)$ is nodal (sign changing). Since $\hat{\sigma}(p)$ is closed and $\hat{\lambda}_1(p)>0$ is isolated, the
 second
  eigenvalue $\hat{\lambda}_2(p)$ is well-defined by
$$\hat{\lambda}_2(p)=\min\{\hat{\lambda}\in\hat{\sigma}(p):\hat{\lambda}>\hat{\lambda}_1(p)\}.$$

For additional eigenvalues, we employ the Ljusternik-Schnirelmann minimax scheme which gives the entire nondecreasing sequence of eigenvalues $\{\hat{\lambda}_k(p)\}_{k\geq 1}$ such that $\hat{\lambda}_k(p)\rightarrow+\infty$. These eigenvalues are known as ``variational eigenvalues'' and depending on the index used in the Ljusternik-Schnirelmann scheme, we can have various such sequences of variational eigenvalues, which all coincide in the first two elements $\hat{\lambda}_1(p)$ and $\hat{\lambda}_2(p)$ defined as described above. For the other elements we do not know if their sequences coincide. Here we use the sequence constructed by using the Fadell-Rabinowitz \cite{13} cohomological index (see Perera \cite{29}). Note that we do not know if the variational eigenvalues exhaust the spectrum $\hat{\sigma}(p)$. We have full knowledge of the spectrum if $N=1$ (ordinary differential equations) and when $p=2$ (linear eigenvalue problem). In the latter case, we have $\hat{\sigma
 }(2)=\{\
 \hat{\lambda}_k(2)\}_{k\geq 1}$ with $0<\hat{\lambda}_1(2)<\hat{\lambda}_2(2)<\ldots<\hat{\lambda}_k(2)\rightarrow+\infty$ ad $k\rightarrow\infty$. The corresponding eigenspaces, denoted by $E(\hat{\lambda}_k(2))$, are linear spaces and we have the following orthogonal direct sum decomposition
$$H^1_0(\overline{\Omega})=\overline{{\underset{\mathrm{k\geq 1}}\oplus}E(\hat{\lambda}_k(2))}.$$

For all $k\in\NN$, each $E(\hat{\lambda}_k(2))$ is finite dimensional, $E(\hat{\lambda}_k(2))\subseteq C^1_0(\overline{\Omega})$ and has the so-called ``Unique Continuation Property'' (``UCP'' for short), that is, if $u\in E(\hat{\lambda}_k(2))$ vanishes on a set of positive measure in $\Omega$, then $u\equiv 0$. For every $k\in\NN$ we define
$$\bar{H}_k=\overset{k}{\underset{\mathrm{i=1}}\oplus}E(\hat{\lambda}_i(2))\ \mbox{and}\ \hat{H}_{k+1}=\overline{{\underset{\mathrm{i\geq k+1}}\oplus}E(\hat{\lambda}_i(2))}=\bar{H}^\perp_k.$$

We have
$$H^1_0(\Omega)=\bar{H}_k\oplus\hat{H}_{k+1}.$$

In this case all eigenvalues admit variational characterizations and we have
\begin{eqnarray}
	\hat{\lambda}_1(2)&=&\inf\left[\frac{||Du||^2_2}{||u||^2_2}:u\in H^1_0(\Omega),u\neq 0\right]\label{eq4}\\
	\hat{\lambda}_k(2)&=&\sup\left[\frac{||Du||^2_2}{||u||^2_2}:u\in \bar{H}_k,u\neq 0\right]\nonumber\\
	&=&\inf\left[\frac{||Du||^2_2}{||u||^2_2}:u\in\hat{H}_k,u\neq 0\right],\ k\geq 2.\label{eq5}
\end{eqnarray}

Again, the infimum in (\ref{eq4}) is realized on the one-dimensional eigenspace $E(\hat{\lambda}_1(2))$, while both the supremum and the infimum in (\ref{eq5}) are realized on $E(\hat{\lambda}_k(2))$.

As a consequence of the UCP, we have the following convenient inequalities.
\begin{lemma}\label{lem2}
	\begin{itemize}
		\item[(a)] If $k\in\NN,\vartheta\in L^{\infty}(\Omega),\vartheta(z)\geq\hat{\lambda}_k(2)$ for almost all $z\in\Omega$, $\vartheta\not\equiv\hat{\vartheta}_k(2)$, then there exists a constant $c_0>0$ such that
		$$||Du||^2_2-\int_{\Omega}\vartheta(z)u^2dz\leq-c_0||u||^2\ \mbox{for all}\ u\in\bar{H}_k.$$
		\item[(b)] If $k\in\NN,\vartheta\in L^{\infty}(\Omega)$, $\vartheta(z)\leq\hat{\lambda}_k(2)$ for almost all $z\in\Omega$, $\vartheta\not\equiv\hat{\lambda}_k(2)$, then there exists a constant $c_1>0$
		$$||Du||^2_2-\int_{\Omega}\vartheta(z)u^2dz\geq c_1||u||^2\ \mbox{for all}\ u\in\hat{H}_k.$$
	\end{itemize}
\end{lemma}

In what follows, let $A_p:W^{1,p}_{0}(\Omega)\rightarrow W^{-1,p'}(\Omega)=W^{1,p}_{0}(\Omega)^*$ $\left(\frac{1}{p}+\frac{1}{p'}=1,1<p<\infty\right)$ be the map defined by
$$\left\langle A_p(u),h\right\rangle=\int_{\Omega}|Du|^{p-2}(Du,Dh)_{\RR^N}dz\ \mbox{for all}\ u,h\in W^{1,p}_{0}(\Omega).$$

By Motreanu, Motreanu and Papageorgiou \cite[p. 40]{21}, we have:
\begin{prop}\label{prop3}
	The map $A_p:W^{1,p}_{0}(\Omega)\rightarrow W^{-1,p'}(\Omega)\ (1<p<\infty)$ is bounded (that is, it maps bounded sets to bounded sets), continuous, strictly monotone (hence maximal monotone, too) and of type $(S)_+$, that is,
	\begin{center}
	``if $u_n\stackrel{w}{\rightarrow}u$ in $W^{1,p}_{0}(\Omega)$ and $\limsup\limits_{n\rightarrow\infty}\left\langle A(u_n),u_n-u\right\rangle\leq 0$, then $u_n\rightarrow n$ in $W^{1,p}_{0}(\Omega)$.''
	\end{center}
\end{prop}
	
	If $p=2$, then $A_2=A\in\mathcal{L}(H^1_0(\Omega),H^{-1}(\Omega))$.
	
	Consider a Carath\'eodory function $f_0:\Omega\times\RR\rightarrow\RR$ such that
	$$|f_0(z,x)|\leq a_0(z)(1+|x|^{r-1})\ \mbox{for almost all}\ z\in\Omega,\ \mbox{all}\ x\in\RR,$$
	with $a_0\in L^{\infty}(\RR)$ and $1<r<p^*$, where $p^*=\left\{\begin{array}{ll}
		\frac{Np}{N-p}&\mbox{if}\ p<N\\
		+\infty&\mbox{if}\ p\geq N
	\end{array}\right.$ (the critical Sobolev exponent). Let $F_0(z,x)=\int^x_0f_0(z,s)ds$ and consider the $C^1$-functional $\varphi_0:W^{1,p}_{0}(\Omega)\rightarrow\RR$ defined by
	$$\varphi_0(u)=\frac{1}{p}||Du||^p_p+\frac{1}{2}||Du||^2_2-\int_{\Omega}F_0(z,u)dz\ \mbox{for all}\ u\in W^{1,p}_{0}(\Omega).$$
	
	The next proposition is a special case of a more general result by Aizicovici, Papageorgiou and Staicu \cite{2}, see also Papageorgiou and R\u adulescu \cite{25,26} for similar results in different spaces. All these results are consequences of the nonlinear regularity theory of Lieberman \cite{19}.
\begin{prop}\label{prop4}
		Let $u_0\in W^{1,p}_{0}(\Omega)$ be a local $C^1_0(\bar{\Omega})$-minimizer of $\varphi_0$, that is, there exists $\rho_0>0$ such that
		$$\varphi_0(u_0)\leq\varphi_0(u_0+h)\ \mbox{for all}\ h\in C^1_0(\overline{\Omega})\ \mbox{with}\ ||h||_{C^1_0(\overline{\Omega})}\leq\rho_0.$$
		Then $u_0\in C^{1,\alpha}_0(\overline{\Omega})$ for some $\alpha\in(0,1)$ and it is also a local $W^{1,p}_{0}(\Omega)$-minimizer of $\varphi_0$, that is, there exists $\rho_1>0$ such that
		$$\varphi_0(u_0)\leq\varphi_0(u_0+h)\ \mbox{for all}\ h\in W^{1,p}_{0}(\Omega)\ \mbox{with}\ ||h||\leq\rho_1.$$
\end{prop}

Finally, we recall some basic definitions and facts from Morse theory (critical groups) which we will use in the sequel.

So, let $X$ be a Banach space, $\varphi\in C^1(X,\RR)$ and $c\in\RR$. We introduce the following sets:
\begin{eqnarray*}
	&&K_{\varphi}=\{u\in X:\varphi'(u)=0\},\\
	&&K^c_{\varphi}=\{u\in K_{\varphi}:\varphi(u)=c\},\\
	&&\varphi^c=\{u\in X:\varphi(u)\leq c\}.
\end{eqnarray*}

Let $(Y_1,Y_2)$ be a topological pair such that $Y_2\subseteq Y_1\subseteq X$ and $k\in\NN_0$. By $H_k(Y_1,Y_2)$ we denote the $k$th relative singular homology group with integer coefficients for the pair $(Y_1,Y_2)$. Given an isolated $u\in K^c_{\varphi}$, the critical groups of $\varphi$ at $u$ are defined by
$$C_k(\varphi,u)=H_k(\varphi^c\cap U,\varphi^c\cap U\backslash\{u\})\ \mbox{for all}\ k\in\NN_0,$$
where $U$ is a neighborhood of $u$ such that $K_{\varphi}\cap \varphi^c\cap U=\{u\}$. The excision property of singular homology implies that the above definition of critical groups is independent of the particular choice of the neighborhood $U$.

Suppose that $\varphi$ satisfies the C-condition and $\inf\varphi(K_{\varphi})>-\infty$. Let $c<\inf\varphi(K_{\varphi})$. The critical groups of $\varphi$ at infinity are defined by
$$C_k(\varphi,\infty)=H_k(X,\varphi^c)\ \mbox{for all}\ k\in\NN_0.$$

The second deformation theorem (see, for example, Gasinski and Papageorgiou \cite[p. 628]{15}), implies that this definition is independent of the choice of the level $c<\inf\varphi(K_{\varphi})$.

In the next section we prove an existence theorem under conditions of resonance both at $\pm\infty$ and at zero.

\section{Existence of Nontrivial Solutions}

The hypotheses on the reaction term $f(z,x)$ are the following:

\smallskip
$H_1:$ $f:\Omega\times\RR\rightarrow\RR$ is a Carath\'eodory function such that
\begin{itemize}
	\item[(i)] for every $\rho>0$, there exists $a_{\rho}\in L^{\infty}(\Omega)_+$ such that
	$$|f(z,x)|\leq a_{\rho}(z)\ \mbox{for almost all}\ z\in\Omega,\ \mbox{all}\ |x|\leq\rho;$$
	\item[(ii)] there exists an integer $m\geq 1$ such that
	$$\lim\limits_{x\rightarrow\pm\infty}\frac{f(z,x)}{|x|^{p-2}x}=\hat{\lambda}_m(p)\ \mbox{uniformly for almost all}\ z\in\Omega;$$
	\item[(iii)] there exists $\tau\in(2,p)$ such that
	$$0<\beta_0\leq\liminf\limits_{x\rightarrow\pm\infty}\frac{f(z,x)x-pF(z,x)}{|x|^{\tau}}\ \mbox{uniformly for almost all}\ z\in\Omega,$$
	where $F(z,x)=\int^x_0f(z,s)ds$;
	\item[(iv)] there exist integer $l\geq 1$ with $d_l\neq m$ $(d_l={\rm dim}\,\bar{H}_l)$, $\delta>0$ and $\eta\in L^{\infty}(\Omega)$ such that
	\begin{eqnarray*}
		&&\hat{\lambda}_l(2)\leq\eta(z)\ \mbox{for almost all}\ z\in\Omega,\ \eta\not\equiv\hat{\lambda}_l(2),\\
		&&\eta(z)x^2\leq f(z,x)x\leq\hat{\lambda}_{l+1}x^2\ \mbox{for almost all}\ z\in\Omega,\ \mbox{all}\ |x|\leq\delta
	\end{eqnarray*}
	and for every $x\neq 0$ the second inequality is strict on a subset of positive Lebesgue measure.
\end{itemize}
\begin{remark}
	Hypothesis $H_1(ii)$ says that asymptotically as $x\rightarrow\pm\infty$, we have resonance with respect to some variational eigenvalue of $(-\Delta_p,W^{1,p}_{0}(\Omega))$. Similarly, hypothesis $H_1(iv)$ permits resonance at zero with respect to the eigenvalue $\hat{\lambda}_{l+1}(2)$ of $(-\Delta,H^1_0(\Omega))$. So, in a sense, we have a double resonance setting.
\end{remark}

Let $\varphi:H^1(\Omega)\rightarrow\RR$ be the energy (Euler) functional for problem (\ref{eq1}) defined by
$$\varphi(u)=\frac{1}{p}||Du||^p_p+\frac{1}{2}||Du||^2_2-\int_{\Omega}F(z,u)dz\ \mbox{for all}\ u\in W^{1,p}_{0}(\Omega).$$
\begin{prop}\label{prop5}
	If hypotheses $H_1(i),(ii),(iii)$ hold, then $\varphi$ satisfies the C-condition.
\end{prop}
\begin{proof}
	Let $\{u_n\}_{n\geq 1}\subseteq W^{1,p}_{0}(\Omega)$ be a sequence such that
	\begin{eqnarray}
		&&|\varphi(u_n)|\leq M_1\ \mbox{for some}\ M_1>0,\ \mbox{all}\ n\in\NN,\label{eq6}\\
		&&(1+||u_n||)\varphi'(u_n)\rightarrow 0\ \mbox{in}\ W^{-1,p'}(\Omega)=W^{1,p}_{0}(\Omega)^*.\label{eq7}
	\end{eqnarray}
	
	By (\ref{eq7}) we have
	\begin{eqnarray}\label{eq8}
		&&\left|\left\langle A_p(u_n),h\right\rangle+\left\langle A(u_n),h\right\rangle-\int_{\Omega}f(z,u_n)hdz\right|\leq\frac{\epsilon_n||h||}{1+||u_n||}\\
		&&\mbox{for all}\ h\in W^{1,p}_{0}(\Omega)\ \mbox{with}\ \epsilon_n\rightarrow 0^+.\nonumber
	\end{eqnarray}
	
	In (\ref{eq8}) we choose $h=u_n\in W^{1,p}_{0}(\Omega)$ and obtain
	\begin{equation}\label{eq9}
		-||Du_n||^p_p-||Du_n||^2_2+\int_{\Omega}f(z,u_n)u_ndz\leq\epsilon_n\ \mbox{for all}\ n\in\NN\,.
	\end{equation}
	
	On the other hand, from (\ref{eq6}) we have
	\begin{equation}\label{eq10}
		||Du_n||^p_p+\frac{p}{2}||Du_n||^2_2-\int_{\Omega}pF(z,u_n)dz\leq pM_1\ \mbox{for all}\ n\in\NN.
	\end{equation}
	
	We add (\ref{eq9}) and (\ref{eq10}) and obtain
	\begin{eqnarray}\label{eq11}
		&&\int_{\Omega}[f(z,u_n)u_n-pF(z,u_n)]dz\leq M_2+\left(1-\frac{p}{2}\right)||Du_n||^2_2\\
		&&\mbox{for some}\ M_2>0,\ \mbox{all}\ n\in\NN\,.\nonumber
	\end{eqnarray}
	
	Hypotheses $H_1(i),(iii)$ imply that we can find $\beta_1\in(0,\beta_0)$ and $c_2>0$ such that
	\begin{equation}\label{eq12}
		\beta_1|x|^{\tau}-c_2\leq f(z,x)x-pF(z,x)\ \mbox{for almost all}\ z\in\Omega,\ \mbox{all}\ x\in\RR\,.
	\end{equation}
	
	Returning to (\ref{eq11}) and using (\ref{eq12}), we have
	\begin{equation}\label{eq13}
		||u_n||^{\tau}_{\tau}\leq c_3(1+||Du_n||^2_2)\ \mbox{for some}\ c_3>0,\ \mbox{all}\ n\in\NN\ (\mbox{recall that}\ \tau >2.)
	\end{equation}

	\begin{claim}
		$\{u_n\}_{n\geq 1}\subseteq W^{1,p}_{0}(\Omega)$ is bounded.
	\end{claim}
	
	Arguing by contradiction, suppose that the claim is not true. By passing to a subsequence if necessary, we have
	\begin{equation}\label{eq14}
		||u_n||\rightarrow\infty\,.
	\end{equation}
	
	Let $y_n=\frac{u_n}{||u_n||},\ n\in\NN$. Then $||y_n||=1$ for all $n\in\NN$ and so we may assume that
	\begin{equation}\label{eq15}
		y_n\stackrel{w}{\rightarrow}y\ \mbox{in}\ W^{1,p}_{0}(\Omega)\ \mbox{and}\ y_n\rightarrow y\ \mbox{in}\ L^p(\Omega).
	\end{equation}
	
	From (\ref{eq8}) we have
\begin{eqnarray}\label{eq16}
	&&\left|\left\langle A_p(y_n),h\right\rangle+\frac{1}{||u_n||^{p-2}}\left\langle A(y_n),h\right\rangle-\int_{\Omega}\frac{N_f(u_n)}{||u_n||^{p-1}}hdz\right|\leq\frac{\epsilon_n||h||}{(1+||u_n||)||u_n||^{p-1}}\\
	&&\mbox{for all}\ n\in\NN\,.\nonumber
\end{eqnarray}

Hypotheses $H_1(i),(ii)$ imply that
\begin{eqnarray}
	&&|f(z,x)|\leq c_4(1+|x|^{p-1})\ \mbox{for almost all}\ z\in\Omega,\ \mbox{all}\ x\in\RR,\ \mbox{some}\ c_4>0,\label{eq17}\\
	&\Rightarrow&\left\{\frac{N_f(u_n)}{||u_n||^{p-1}}\right\}_{n\geq 1}\subseteq L^{p'}(\Omega)\ \mbox{is bounded}.\label{eq18}
\end{eqnarray}

In (\ref{eq16}) we choose $h=y_n-y\in W^{1,p}_{0}(\Omega)$, pass to the limit as $n\rightarrow\infty$ and use (\ref{eq14}), (\ref{eq15}), (\ref{eq18}) and the fact that $p>2$. Then
\begin{eqnarray}\label{eq19}
	&&\lim\limits_{n\rightarrow\infty}\left\langle A_p(y_n),y_n-y\right\rangle=0,\nonumber\\
	&\Rightarrow&y_n\rightarrow y\ \mbox{in}\ W^{1,p}_{0}(\Omega)\ \mbox{(see Proposition \ref{prop3}), hence}\ ||y||=1.
\end{eqnarray}

From (\ref{eq13}) we have
\begin{eqnarray*}
	 &&||y_n||^{\tau}_{\tau}\leq\frac{c_3}{||u_n||^{\tau}}+\frac{c_3}{||u_n||^{\tau-2}}||Dy_n||^2_2\leq\frac{c_5}{||u_n||^{\tau-2}}\ \mbox{for some}\ c_5>0,\ \mbox{all}\ n\geq n_0\geq 1,\\
	&\Rightarrow&y_n\rightarrow 0\ \mbox{in}\ L^{\tau}(\Omega)\ \mbox{as}\ n\rightarrow\infty\ (\mbox{see (\ref{eq14}) and recall that}\ \tau>2),\\
	&\Rightarrow&y=0\ \mbox{(see (\ref{eq15})), a contradiction to (\ref{eq19})}.
\end{eqnarray*}

This proves the claim.

Because of Claim 1 we may assume that
\begin{equation}\label{eq20}
	u_n\stackrel{w}{\rightarrow}u\ \mbox{in}\ W^{1,p}_{0}(\Omega)\ \mbox{and}\ u_n\rightarrow u\ \mbox{in}\ L^p(\Omega).
\end{equation}

From (\ref{eq17}) we see that
\begin{equation}\label{eq21}
	\{N_f(u_n)\}_{n\geq 1}\subseteq L^{p'}(\Omega)\ \mbox{is bounded}.
\end{equation}

So, if in (\ref{eq8}) we choose $h=y_n-y\in W^{1,p}_{0}(\Omega)$, pass to the limit as $n\rightarrow\infty$ and use (\ref{eq20}), (\ref{eq21}), then
\begin{eqnarray*}
	&&\lim\limits_{n\rightarrow\infty}\left[\left\langle A_p(u_n),u_n-u\right\rangle+\left\langle A(u_n),u_n-u\right\rangle\right]=0,\\
	&\Rightarrow&\limsup\limits_{n\rightarrow\infty}\left[\left\langle A_p(u_n),u_n-u\right\rangle+\left\langle A(u),u_n-u\right\rangle\right]\leq 0\\
	&&(\mbox{since}\ A(\cdot)\ \mbox{is monotone})\\
	&\Rightarrow&\limsup\limits_{n\rightarrow\infty}\left\langle A_p(u_n),u_n-u\right\rangle\leq 0\ (\mbox{see (\ref{eq20})}),\\
	&\Rightarrow&u_n\rightarrow u\ \mbox{in}\ W^{1,p}_{0}(\Omega)\ (\mbox{see Proposition \ref{prop3}}),\\
	&\Rightarrow&\varphi\ \mbox{satisfies the C-condition.}
\end{eqnarray*}
\end{proof}

We can have two approaches in the proof of the existence theorem. We present both, because  we believe that the particular tools used in each of them are of independent interest and can be used in different circumstances.

In the first approach we compute directly the critical groups at infinity of the energy functional $\varphi$. Note that Proposition \ref{prop5} permits this fact.

\begin{prop}\label{prop6}
	If hypotheses $H_1(i),(ii),(iii)$ hold, then $C_m(\varphi,\infty)\neq 0$.
\end{prop}
\begin{proof}
	Let $\lambda\in(\hat{\lambda}_m(p),\hat{\lambda}_{m+1}(p))\backslash\hat{\sigma}(p)$ and consider the $C^2$-functional $\psi:W^{1,p}_{0}(\Omega)\rightarrow\RR$ defined by
	$$\psi(u)=\frac{1}{p}||Du||^p_p-\frac{\lambda}{p}||u||^p_p\ \mbox{for all}\ u\in W^{1,p}_{0}(\Omega).$$
	
	We consider the homotopy $h(t,u)$ defined by
	$$h(t,u)=(1-t)\varphi(u)+t\psi(u)\ \mbox{for all}\ (t,u)\in[0,1]\times W^{1,p}_{0}(\Omega).$$
	\begin{claim}
		There exist $\eta\in\RR$ and $\hat{\delta}>0$ such that
		$$h(t,u)\leq\eta\Rightarrow(1+||u||)||h'_u(t,u)||_*\geq\hat{\delta}\ \mbox{for all}\ t\in[0,1].$$
	\end{claim}
	
	We argue indirectly. So, suppose that the claim is not true. Since $h(\cdot,\cdot)$ maps bounded sets to bounded sets, we can find $\{t_n\}_{n\geq 1}\subseteq[0,1]$ and $\{u_n\}_{n\geq 1}\subseteq W^{1,p}_{0}(\Omega)$ such that
	\begin{equation}\label{eq22}
		t_n\rightarrow t,\, ||u_n||\rightarrow\infty,\, h(t_n,u_n)\rightarrow-\infty\ \mbox{and}\ (1+||u_n||)h'_u(t_n,u_n)\rightarrow 0\ \mbox{in}\ W^{-1,p'}(\Omega).
	\end{equation}
	
	From the last convergence in (\ref{eq22}), we have
	\begin{eqnarray}\label{eq23}
		&&\left|\left\langle A_p(u_n),h\right\rangle+(1-t_n)\left\langle A(u_n),h\right\rangle-(1-t_n)\int_{\Omega}f(z,u_n)hdz-t_n\int_{\Omega}\lambda|u_n|^{p-2}u_nhdz\right|\nonumber\\
		&&\leq\frac{\epsilon_n||h||}{1+||u_n||}\ \mbox{for all}\ h\in W^{1,p}_{0}(\Omega)\ \mbox{with}\ \epsilon_n\rightarrow 0^+.
	\end{eqnarray}
	
	Let $y_n=\frac{u_n}{||u_n||},\ n\in\NN$. Then $||y_n||=1$ for all $n\in\NN$ and so we may assume that
	\begin{equation}\label{eq24}
		y_n\stackrel{w}{\rightarrow}y\ \mbox{in}\ W^{1,p}_{0}(\Omega)\ \mbox{and}\ y_n\rightarrow y\ \mbox{in}\ L^p(\Omega).
	\end{equation}
	
	From (\ref{eq23}) we have
	\begin{eqnarray}\label{eq25}
		&&\left|\left\langle A_p(y_n),h\right\rangle+\frac{1-t_n}{||u_n||^{p-2}}\left\langle A(y_n),h\right\rangle-(1-t_n)\int_{\Omega}\frac{N_f(u_n)}{||u_n||^{p-1}}hdz-t_n\int_{\Omega}\lambda|y_n|^{p-2}y_nhdz\right|\nonumber\\
		&&\leq\frac{\epsilon_n||h||}{(1+||u_n||)||u_n||^{p-1}}\ \mbox{for all}\ n\in\NN\,.
	\end{eqnarray}
	
	From (\ref{eq17}) and (\ref{eq24}), we see that
	$$\left\{\frac{N_f(u_n)}{||u_n||^{p-1}}\right\}_{n\geq 1}\subseteq L^{p'}(\Omega)\ \mbox{is bounded}.$$
	
	Hence, by passing to a subsequence if necessary and using hypothesis $H_1(ii)$ we obtain
	\begin{equation}\label{eq26}
		\frac{N_f(u_n)}{||u_n||^{p-1}}\stackrel{w}{\rightarrow}\hat{\lambda}_m(p)|y|^{p-2}y\ \mbox{in}\ L^{p'}(\Omega)
	\end{equation}
	(see Filippakis and Papageorgiou \cite{14}).
	
	In (\ref{eq25}) we choose $h=y_n-y\in W^{1,p}_{0}(\Omega)$, pass to the limit as $n\rightarrow\infty$ and use (\ref{eq22}), (\ref{eq24}), (\ref{eq26}) and the fact that $2<p$. Then
	\begin{eqnarray}\label{eq27}
		&&\lim\limits_{n\rightarrow\infty}\left\langle A_p(y_n),y_n-y\right\rangle=0,\nonumber\\
		&\Rightarrow&y_n\rightarrow y\ \mbox{in}\ W^{1,p}_{0}(\Omega)\ \mbox{(see Proposition \ref{prop3}), so}\ ||y||=1.
	\end{eqnarray}
	
	We return to (\ref{eq25}), pass to the limit as $n\rightarrow\infty$ and use (\ref{eq26}), (\ref{eq27}). We obtain
	\begin{eqnarray}\label{eq28}
		&&\left\langle A_p(y),h\right\rangle=\int_{\Omega}\lambda_t|y|^{p-2}yhdz\ \mbox{for all}\ h\in W^{1,p}_{0}(\Omega),\ \mbox{with}\ \lambda_t=(1-t)\hat{\lambda}_m(p)+t\lambda\nonumber\\
		&\Rightarrow&-\Delta_py(z)=\lambda_t|y(z)|^{p-2}y(z)\ \mbox{for almost all}\ z\in\Omega,\ y|_{\partial\Omega}=0.
	\end{eqnarray}
	
	If $\lambda_t\notin\hat{\sigma}(p)$, then from (\ref{eq28}) it follows that $y=0$, a contradiction (see (\ref{eq27})).
	
	If $\lambda_t\in\hat{\sigma}(p)$, then for $E=\{z\in\Omega:y(z)\neq 0\}$ we have $|E|_N>0$. Hence
	\begin{eqnarray}\label{eq29}
		&&|u_n(z)|\rightarrow+\infty\ \mbox{for almost all}\ z\in\Omega,\nonumber\\
		 &\Rightarrow&\liminf\limits_{n\rightarrow\infty}\frac{f(z,u_n(z))u_n(z)-pF(z,u_n(z))}{|u_n(z)|^{\tau}}\geq\beta_0>0\ \mbox{for almost all}\ z\in E.
	\end{eqnarray}
	
	From (\ref{eq29}), hypothesis $H_1(iii)$ and Fatou's lemma, we have
	\begin{equation}\label{eq30}
		\liminf\limits_{n\rightarrow\infty}\frac{1}{||u_n||^{\tau}}\int_E[f(z,u_n)u_n-pF(z,u_n)]dz>0.
	\end{equation}
	
	Note that hypothesis $H_1(iii)$ imply that we can find $M_3>0$ such that
	\begin{equation}\label{eq31}
	f(z,x)x-	pF(z,x)\geq 0\ \mbox{for almost all}\ z\in\Omega,\ \mbox{all}\ |x|\geq M_3.
	\end{equation}
	
	Then we have
	\begin{eqnarray}\label{eq32}
		&&\frac{1}{||u_n||^{\tau}}\int_{\Omega}[f(z,u_n)u_n-pF(z,u_n)]dz\nonumber\\
		&=&\frac{1}{||u_n||^{\tau}}\int_{\Omega\cap\{|u_n|\geq M_3\}}[f(z,u_n)u_n-pF(z,u_n)]dz+\nonumber\\
		&&\hspace{1cm}\frac{1}{||u_n||^{\tau}}\int_{\Omega\cap\{|u_n|<M_3\}}[f(z,u_n)u_n-pF(z,u_n)]dz\nonumber\\
		&\geq&\frac{1}{||u_n||^{\tau}}\int_{E\cap\{|u_n|\geq M_3\}}[f(z,u_n)u_n-pF(z,u_n)]dz-\frac{c_6}{||u_n||^{\tau}}\nonumber\\
		&&\mbox{for some}\ c_6>0,\ \mbox{all}\ n\in\NN\ (\mbox{see (\ref{eq31}) and hypothesis}\ H_1(i))\nonumber\\
		&\geq&\frac{1}{||u_n||^{\tau}}\int_{E}[f(z,u_n)u_n-pF(z,u_n)]dz-\frac{c_7}{||u_n||^{\tau}}\ \mbox{for some}\ c_7>0,\ \mbox{all}\ n\in\NN\nonumber\\
		&&(\mbox{see hypothesis}\ H_1(i)),\nonumber\\
		\Rightarrow&&\liminf\limits_{n\rightarrow\infty}\frac{1}{||u_n||^{\tau}}\int_{\Omega}[f(z,u_n)u_n-pF(z,u_n)]dz>0\ (\mbox{see (\ref{eq30})}).
	\end{eqnarray}
	
	From the third convergence in (\ref{eq22}), we see that we can find $n_0\in\NN$ such that
	\begin{eqnarray}\label{eq33}
		 &&||Du_n||^p_p+\frac{(1-t_n)p}{2}||Du_n||^2_2-(1-t_n)\int_{\Omega}pF(z,u_n)dz-t_n\int_{\Omega}\lambda|u_n|^pdz\leq-1\\
		&&\mbox{for all}\ n\geq n_0.\nonumber
	\end{eqnarray}
	
	In (\ref{eq23}) we choose $h=u_n\in W^{1,p}_0(\Omega)$. Then
	\begin{equation}\label{eq34}
		 -||Du_n||^p_p-(1-t_n)||Du_n||^2_2+(1-t_n)\int_{\Omega}f(z,u_n)u_ndz+t_n\int_{\Omega}\lambda|u_n|^pdz\leq\epsilon_m\ \mbox{for all}\ n\in\NN\,.
	\end{equation}
	
	Since $\epsilon_n\rightarrow 0^+$, by choosing $n_0\in\NN$ even bigger if necessary, we can get
	\begin{equation}\label{eq35}
		\epsilon_n\in(0,1)\ \mbox{for all}\ n\geq n_0.
	\end{equation}
	
	We add (\ref{eq33}), (\ref{eq34}) and use (\ref{eq35}). Then
	$$(1-t_n)\int_{\Omega}[f(z,u_n)u_n-pF(z,u_n)]dz\leq(1-t_n)\left(1-\frac{p}{2}\right)||Du_n||^2_2.$$
	
	We may assume that $t_n\neq 1$ for all $n\in\NN$. Otherwise $t=1$ and so $\lambda_t=\lambda\notin\sigma(p)$, hence $y=0$, a contradiction to (\ref{eq27}). Then
	\begin{equation}\label{eq36}
		 \frac{1}{||u_n||^{\tau}}\int_{\Omega}[f(z,u_n)u_n-pF(z,u_n)]dz\leq\left(\frac{p}{2}-1\right)\frac{1}{||u_n||^{\tau-2}}||Dy_n||^2_2\ \mbox{for all}\ n\in\NN\,.
	\end{equation}
	
	Since $p>\tau>2$,  it follows from (\ref{eq22}) and (\ref{eq27}) that
	$$\limsup\limits_{n\rightarrow\infty}\int_{\Omega}[f(z,u_n)u_n-pF(z,u_n)]dz\leq 0,$$
	which contradicts (\ref{eq32}).
	
	This proves the claim.
	
\smallskip
	In fact the above argument with minor changes, shows that for every $t\in[0,1]$, $h(t,\cdot)$ satisfies the C-condition. So Theorem 5.1.12, p. 334, of Chang \cite{8} (see also Liang and Su \cite[Proposition 3.2]{18}) implies that
	\begin{eqnarray}\label{eq37}
		&&C_k(h(0,\cdot),\infty)=C_k(h(1,\cdot),\infty)\ \mbox{for all}\ k\in\NN_0,\nonumber\\
		&\Rightarrow&C_k(\varphi,\infty)=C_k(\psi,\infty)\ \mbox{for all}\ k\in\NN_0.
	\end{eqnarray}
	
	Since $\lambda\notin\hat{\sigma}(p)$, we have $K_{\psi}=\{0\}$ and so $C_k(\psi,\infty)=C_k(\psi,0)$ for all $k\in\NN_0$. Hence
	\begin{equation}\label{eq38}
		C_k(\varphi,\infty)=C_k(\psi,0)\ \mbox{for all}\ k\in\NN_0\ (\mbox{see (\ref{eq37})}).
	\end{equation}
	
	But by Proposition 1.1 of Perera \cite{29}, we have $C_m(\psi,0)\neq 0$. So,
	$$C_m(\varphi,\infty)\neq 0\ (\mbox{see (\ref{eq38})}).$$
\end{proof}

In the second approach we avoid the computation of the critical groups of $\varphi$ at infinity. Instead we use the following result which is essentially due to Perera \cite[Lemma 4.1]{29}, adapted to our setting here.
\begin{prop}\label{prop7}
	If hypotheses $H_1(i),(ii),(iii)$, then there exist $r>0$ and $\varphi_0\in C^1(W^{1,p}_0(\Omega))$ such that
	$$\varphi_0(u)=\left\{\begin{array}{ll}
		\varphi(u)&\mbox{if}\ ||u||\leq r\\
		\psi(u)&\mbox{if}\ ||u||\geq 2^{\frac{1}{p}}r,
	\end{array}\right.$$
	$K_{\varphi_0}=K_{\varphi}$ and $C_m(\varphi_0,\infty)\neq 0$.
\end{prop}
\begin{proof}
	Let $\psi\in C^2(W^{1,p}_0(\Omega))$ be as in the proof of Proposition \ref{prop6}. Also let $\tau:W^{1,p}_0(\Omega)\rightarrow\RR$ be the $C^1$-functional defined by
	$$\tau(u)=\int_{\Omega}F(z,u)dz-\frac{\lambda}{p}||u||^p_p-\frac{1}{2}||Du||^2_2\ \mbox{for all}\ u\in W^{1,p}_0(\Omega).$$
	
	Evidently, we have
	\begin{equation}\label{eq39}
		\varphi(u)=\psi(u)-\tau(u)\ \mbox{for all}\ u\in W^{1,p}_0(\Omega).
	\end{equation}
	
	Since $\lambda\notin\hat{\sigma}(p)$, the functional $\psi$ satisfies the C-condition and so
	$$\mu=\inf\left[||\psi'(u)||_*:u\in W^{1,p}_0(\Omega),||u||=1\right]>0.$$
	
	We have
	$$\psi'(u)=A_p(u)-\lambda|u|^{p-2}u,$$
	hence the $(p-1)$-homogeneity of $\psi'(\cdot)$ implies that
	\begin{equation}\label{eq40}
		\inf[||\psi'(u)||_*:u\in W^{1,p}_0(\Omega),||u||=r]=r^{p-1}\mu>0\ (r>0).
	\end{equation}
	
	Since $\lambda>\hat{\lambda}_m(p)$ and $p>2$, it follows that
	\begin{equation}\label{eq41}
		\limsup\limits_{||u||\rightarrow\infty}\frac{\tau'(u)}{||u||^{p-1}}\leq 0.
	\end{equation}
	
	We have
	\begin{eqnarray}\label{eq42}
		&&\varphi'(u)=\psi'(u)-\tau'(u)\ \mbox{for all}\ u\in W^{1,p}_0(\Omega)\ (\mbox{see (\ref{eq39})}),\nonumber\\
		&\Rightarrow&\varphi'(u)>0\ \mbox{and}\ \varphi'(u)+\tau'(u)>0\ \mbox{for all}\ ||u||>r\ (\mbox{see (\ref{eq40}),\ (\ref{eq41})}).
	\end{eqnarray}
	
	Let $\xi:\RR_+\rightarrow[0,1]$ be a $C^1$-function such that $|\xi'(t)|\leq 1$ for all $t\geq 0$ and
	\begin{equation}\label{eq43}
		\xi(t)=\left\{\begin{array}{ll}
			0&\mbox{if}\ t\in[0,1]\\
			1&\mbox{if}\ t\geq 2.
		\end{array}\right.
	\end{equation}
	
	We define
	$$\varphi_0(u)=\varphi(u)+\xi\left(\frac{||u||^p}{r^p}\right)\tau(u)\ \mbox{for all}\ u\in W^{1,p}_0(\Omega).$$
	
	Evidently, $\varphi_0\in C^1(W^{1,p}_0(\Omega))$ and from (\ref{eq42}), (\ref{eq43}) it follows that
	\begin{eqnarray}
		&&\varphi_0(u)=\left\{\begin{array}{ll}
			\varphi(u)&\mbox{if}\ ||u||\leq r\\
			\psi(u)&\mbox{if}\ ||u||\geq 2^{1/p}r
		\end{array}\right.\label{eq44}\\
		\mbox{and}&&K_{\varphi_0}=K_{\varphi}\subseteq\bar{B}_r.\label{eq45}
	\end{eqnarray}
	
	Moreover, by (\ref{eq44}), (\ref{eq45}) it is clear that
	\begin{eqnarray*}
		&&C_k(\varphi_0,\infty)=C_k(\psi,\infty)\ \mbox{for all}\ k\in\NN_0,\\
		&\Rightarrow&C_k(\varphi_0,\infty)=C_k(\psi,0)\ \mbox{for all}\ k\in\NN\ (\mbox{since}\ K_{\psi}=\{0\},\ \mbox{recall}\ \lambda\notin\hat{\sigma}(p)),\\
		&\Rightarrow&C_m(\varphi_0,\infty)\neq 0\ (\mbox{see Proposition 1.1 of Perera \cite{29}}).
	\end{eqnarray*}
\end{proof}

Next, we turn our attention to the critical groups of $\varphi$ at the origin. To compute them we only need a subcritical growth on $f(z,\cdot)$ and the behavior of $f(z,\cdot)$ near zero. So, we introduce the following weaker set of hypotheses on $f(z,x)$:

\smallskip
$H_0:$ $f:\Omega\times\RR\rightarrow\RR$ is a Carath\'eodory function such that
\begin{itemize}
	\item[(i)] $|f(z,x)|\leq a(z)(1+|x|^{r-1})$ for almost all $z\in\Omega$, all $x\in\RR$ with $a\in L^{\infty}(\Omega)_+,\ p\leq r<p^*$;
	\item[(ii)] there exist $l\in\NN,\delta>0$ and $\eta\in L^{\infty}(\Omega)$ such that
	\begin{eqnarray*}
		&&\hat{\lambda}_l(2)\leq\eta(z)\ \mbox{for almost all}\ z\in\Omega,\eta\not\equiv\hat{\lambda}_l(2),\\
		&&\eta(z)x^2\leq f(z,x)x\leq\hat{\lambda}_{l+1}(2)x^2\ \mbox{for almost all}\ z\in\Omega,\ \mbox{all}\ |x|\leq\delta
	\end{eqnarray*}
	and for every $x\neq 0$ the second inequality is strict on a set of positive Lebesgue measure.
\end{itemize}
\begin{prop}\label{prop8}
	If hypotheses $H_0$ hold and the functional $\varphi$ satisfies the C-condition, then $C_k(\varphi,0)=\delta_{k,d_l}\ZZ$ for all $k\in\NN_0$ with $d_l={\rm dim}\, \bar{H}_l.$
\end{prop}
\begin{proof}
	We consider the $C^2$-functional $\hat{\psi}:H^1_0(\Omega)\rightarrow\RR$ defined by
	$$\hat{\psi}(u)=\frac{1}{2}||Du||^2_2-\int_{\Omega}F(z,u)dz\ \mbox{for all}\ u\in H^1_0(\Omega).$$
	
	We set $\psi=\hat{\psi}|_{W^{1,p}_0(\Omega)}$ (recall that $p>2$).
	\begin{claim}
		$C_k(\psi,0)=\delta_{k,d_l}\ZZ$ for all $k\in\NN_0$.
	\end{claim}
	
	To prove this claim, let $\vartheta\in(\hat{\lambda}_l(2),\hat{\lambda}_{l+1}(2))$ and consider the $C^2$-functional $\tau:H^1_0(\Omega)\rightarrow\RR$ defined by
	$$\tau(u)=\frac{1}{2}||Du||^2_2-\frac{\vartheta}{2}||u||^2_2\ \mbox{for all}\ u\in H^1_0(\Omega).$$
	
	We consider the homotopy $h(t,u)$ defined by
	$$h(t,u)=(1-t)\hat{\psi}(u)+t\tau(u)\ \mbox{for all}\ (t,u)\in[0,1]\times H^1_0(\Omega).$$
	
	First consider $t\in\left(0,1\right]$. Let $u\in C^1_0(\overline{\Omega})$ with $||u||_{C^1_0(\overline{\Omega})}\leq\delta$ where $\delta>0$ is as in hypothesis $H_0(ii)$. Let $\left\langle \cdot,\cdot\right\rangle_0$ denote the duality brackets for the pair $(H^{-1}(\Omega),H^1_0(\Omega))$. Then we have
	\begin{equation}\label{eq46}
		\left\langle h'_u(t,u),v\right\rangle=(1-t)\left\langle \hat{\psi}'(u),v\right\rangle_0+t\left\langle \tau'(u),v\right\rangle_0\ \mbox{for all}\ v\in H^1_0(\Omega).
	\end{equation}
	
	Recall that $\bar{H}_l=\overset{l}{\underset{\mathrm{k=1}}\oplus}E(\hat{\lambda}_k(2)),\ \hat{H}_{l+1}=\bar{H}^\perp_l=\overline{{\underset{\mathrm{k\geq l+1}}\oplus}E(\hat{\lambda}_k(2))}$ and consider the orthogonal direct sum decomposition
	$$H^1_0(\Omega)=\bar{H}_l\oplus\hat{H}_{l+1}.$$
	
	So, every $u\in H^1_0(\Omega)$ admits a unique sum decomposition
	$$u=\bar{u}+\hat{u}\ \mbox{with}\ \bar{u}\in\bar{H}_l,\ \hat{u}\in\hat{H}_{l+1}.$$
	
	In (\ref{eq46}) we choose $v=\hat{u}-\bar{u}$. Exploiting the orthogonality of the component spaces, we have
	\begin{equation}\label{eq47}
		\left\langle \hat{\psi}(u),\hat{u}-\bar{u}\right\rangle_0=||D\hat{u}||^2_2-||D\bar{u}||^2_2-\int_{\Omega}f(z,u)(\hat{u}-\bar{u})dz.
	\end{equation}
	
	Hypothesis $H_0(ii)$ implies that
	\begin{equation}\label{eq48}
		\eta(z)\leq\frac{f(z,x)}{x}\leq\hat{\lambda}_{l+1}(2)\ \mbox{for almost all}\ z\in\Omega,\ \mbox{all}\ 0<|x|\leq\delta
	\end{equation}
	and the second inequality is for every $x\neq 0$ strict on a set of positive Lebesgue measure. Set $y=\hat{u}-\bar{u}$. Then
	\begin{eqnarray}\label{eq49}
		&f(z,u)(\hat{u}-\bar{u})=f(z,u)y&=\frac{f(z,u)}{u}uy\nonumber\\
		&&\leq\left\{\begin{array}{ll}
			\hat{\lambda}_{l+1}(2)(\hat{u}^2-\bar{u}^2)&\mbox{if}\ uy\geq 0\\
			\eta(z)(\hat{u}^2-\bar{u}^2)&\mbox{if}\ uy<0\ \mbox{(see (\ref{eq48}))}
		\end{array}\right.\nonumber\\
		&&\leq\hat{\lambda}_{l+1}(2)\hat{u}^2-\eta(z)\bar{u}^2\ \mbox{for almost all}\ z\in\Omega .
	\end{eqnarray}
	
	Returning to (\ref{eq47}) and using (\ref{eq49}), we obtain
	\begin{eqnarray}\label{eq50}
		&&\left\langle \hat{\psi}'(u),\hat{u}-\bar{u}\right\rangle_0\geq||D\hat{u}||^2_2-\hat{\lambda}_{l+1}(2)||\hat{u}||^2_2-\left[||D\bar{u}||^2_2-\hat{\lambda}_l(2)||\bar{u}||^2_2\right]\geq 0\\
		&&(\mbox{see hypothesis}\ H_0(ii)\ \mbox{and (\ref{eq5})}).\nonumber
	\end{eqnarray}
	
	Also using Lemma \ref{lem2}, we have
	\begin{equation}\label{eq51}
		\left\langle \tau'(u),\hat{u}-\bar{u}\right\rangle_0=||D\hat{u}||^2_2-\vartheta||\hat{u}||^2_2-[||D\bar{u}||^2_2-\vartheta||\bar{u}||^2_2]\geq c_9||u||^2\ \mbox{for some}\ c_9>0.
	\end{equation}
	
	So, if we use (\ref{eq50}), (\ref{eq51}) in (\ref{eq46}), then
	$$\left\langle h'_u(t,u),\hat{u}-\bar{u}\right\rangle\geq tc_9||u||^2>0\ \mbox{for all}\ t\in\left(0,1\right].$$
	
	Standard regularity theory implies that
	$$K_{h(t,\cdot)}\subseteq C^1_0(\overline{\Omega})\ \mbox{for all}\ t\in[0,1].$$
	
	Therefore we infer that for all $t\in\left(0,1\right],\ u=0$ is isolated in $K_{h(t,\cdot)}$.
	
	We have $h(0,\cdot)=\hat{\psi}(\cdot)$. Next, we show that $0\in K_{\hat{\psi}}$ s isolated. Arguing by contradiction, suppose that we could find $\{u_n\}_{n\geq 1}\subseteq H^1_0(\Omega)$ such that
	\begin{equation}\label{eq52}
		u_n\rightarrow 0\ \mbox{in}\ H^1(\Omega)\ \mbox{and}\ \hat{\psi}'(u_n)=0\ \mbox{for all}\ n\in\NN_0.
	\end{equation}
	
	From the equation in (\ref{eq52}) we have
	\begin{equation}\label{eq53}
		-\Delta u_n(z)=f(z,u_n(z))\ \mbox{for almost all}\ z\in\Omega,\ u_n|_{\partial\Omega}=0,\ n\in\NN\,.
	\end{equation}
	
	From (\ref{eq53}) and standard regularity theory (see, for example, Gasinski and Papageorgiou \cite[pp. 737-738]{15}), we can find $\alpha\in(0,1)$ and $c_{10}>0$ such that
	\begin{equation}\label{eq54}
		u_n\in C^{1,\alpha}_0(\overline{\Omega})\ \mbox{and}\ ||u_n||_{C^{1,\alpha}_0(\overline{\Omega})}\leq c_{10}\ \mbox{for all}\ n\in\NN\,.
	\end{equation}
	
	Exploiting the compact embedding of $C^{1,\alpha}_0(\overline{\Omega})$ into $C^1(\overline{\Omega})$ and using (\ref{eq54}) and (\ref{eq52}), we obtain
	$$u_n\rightarrow 0\ \mbox{in}\ C^1_0(\overline{\Omega}).$$
	
	Therefore we can find $n_0\in\NN$ such that
	\begin{eqnarray}\label{eq55}
		&&|u_n(z)|\leq\delta\ \mbox{for all}\ n\geq n_0,\ \mbox{all}\ z\in\overline{\Omega},\nonumber\\
		&\Rightarrow&\eta(z)u_n(z)^2\leq f(z,u_n(z))u_n(z)\leq\hat{\lambda}_{l+1}(2)u_n(z)^2\\
		&&\mbox{for almost all}\ z\in\Omega,\ \mbox{all}\ n\geq n_0\ (\mbox{see hypothesis }H_0(ii)).\nonumber
	\end{eqnarray}
	
	Then from (\ref{eq54}) and the previous argument, we have
	\begin{equation}\label{eq56}
		f(z,u_n(z))(\hat{u}_n-\bar{u}_n)(z)\leq\hat{\lambda}_{l+1}(2)\hat{u}_n(z)^2-\eta(z)\bar{u}_n(z)^2\ \mbox{for almost all}\ z\in\Omega,\ \mbox{all}\ n\geq n_0.
	\end{equation}
	
	From (\ref{eq53}) we have
	$$\left\langle A(u_n),v\right\rangle=\int_{\Omega}f(z,u_n)vdz\ \mbox{for all}\ v\in H^1_0(\Omega).$$
	
	Choosing $v=\hat{u}_n-\bar{u}_n\in H^1_0(\Omega)$, we obtain
	\begin{eqnarray*}
		&&\int_{\Omega}(Du_n,D\hat{u}_n-D\bar{u}_n)_{\RR^N}dz\\
		&&=||D\hat{u}_n||^2_2-||D\bar{u}_n||^2_2\ \mbox{(from the orthogonality of the component spaces)}\\
		&&=\int_{\Omega}f(z,u_n)(\hat{u}_n-\bar{u}_n)dz\\
		&&\leq\int_{\Omega}\left[\hat{\lambda}_{l+1}(2)\hat{u}_n^2-\eta(z)\bar{u}_n^2\right]dz\ (\mbox{see (\ref{eq56})}),\\
		&\Rightarrow 0&\leq||D\hat{u}_n||^2_2-\hat{\lambda}_{l+1}(2)||\hat{u}_n||^2_2\leq||D\hat{u}_n||^2_2-\int_{\Omega}\eta(z)\bar{u}_n^2dz\leq-c_{11}||\bar{u}_n||^2\\
		&&\mbox{for some}\ c_{11}>0,\ \mbox{all}\ n\geq n_0\ (\mbox{see (\ref{eq5}) and Lemma \ref{lem2}(a)})\\
		&\Rightarrow&\bar{u}_n=0\ \mbox{and}\ \hat{u}_n\in E(\hat{\lambda}_{l+1}(2))\ \mbox{for all}\ n\in\NN\,.
	\end{eqnarray*}
	
	Then $u_n=\hat{u}_n$ for all $n\geq n_0$ and the UCP implies that
	\begin{equation}\label{eq57}
		u_n(z)\neq 0\ \mbox{for almost all}\ z\in\Omega,\ \mbox{all}\ n\geq n_0.
	\end{equation}
	
	From (\ref{eq53}) and (\ref{eq57}), we have
	$$\hat{\lambda}_{l+1}(2)||u_n||^2_2=\int_{\Omega}f(z,u_n)u_ndz<\hat{\lambda}_{l+1}(2)||u_n||^2_2\ \mbox{for all}\ n\geq n_0\ (\mbox{see hypothesis}\ H_0(ii)),$$
	a contradiction. Therefore $0\in K_{\hat{\psi}}$ is isolated and we can conclude that
	$$0\in K_{h(t,\cdot)}\ \mbox{is isolated for all}\ t\in[0,1].$$
	
	So, Theorem 5.2 of Corvellec and Hantoute \cite{11} implies that
	\begin{eqnarray}\label{eq58}
		&&C_k(\hat{\psi},0)=C_k(\tau,0)\ \mbox{for all}\ k\in\NN_0,\nonumber\\
		&\Rightarrow&C_k(\hat{\psi},0)=\delta_{k,d_l}\ZZ\ \mbox{for all}\ k\in\NN_0
	\end{eqnarray}
	(see Motreanu, Motreanu and Papageorgiou \cite[Theorem 6.51, p. 155]{21}).
	
	Since $W^{1,p}_0(\Omega)$ is dense in $H^1_0(\Omega)$, it follows that
	\begin{eqnarray}\label{eq59}
		&&C_k(\hat{\psi},0)=C_k(\psi,0)\ \mbox{for all}\ k\in\NN_0\ (\mbox{see Palais \cite{22} and Chang \cite[p. 14]{7}})\nonumber\\
		&\Rightarrow&C_k(\psi,0)=\delta_{k,d_l}\ZZ\ \mbox{for all}\ k\in\NN_0.
	\end{eqnarray}
	
	We have
	\begin{eqnarray}
		&&|\varphi(u)-\psi(u)|\leq\frac{1}{p}||u||^p,\label{eq60}\\
		&&|\left\langle \varphi'(u)-\psi'(u),v\right\rangle|\leq c_{12}||u||^{p-1}||v||\ \mbox{for some}\ c_{12}>0,\ \mbox{all}\ v\in H^1_0(\Omega),\nonumber\\
		&\Rightarrow&||\varphi'(u)-\psi'(u)||_*\leq c_{12}||u||^{p-1}\label{eq61}
	\end{eqnarray}
	
	Then (\ref{eq60}), (\ref{eq61}) and the $C^1$-continuity of the critical groups (see Theorem 5.1 of Corvellec and Hantoute \cite{11}), imply that
	\begin{eqnarray*}
		&&C_k(\varphi,0)=C_k(\psi,0)\ \mbox{for all}\ k\in\NN_0,\\
		&\Rightarrow&C_k(\varphi,0)=\delta_{k,d_l}\ZZ\ \mbox{for all}\ k\in\NN_0\ (\mbox{see (\ref{eq59})})
	\end{eqnarray*}
\end{proof}

Now we are ready  for the existence theorem.
\begin{theorem}\label{th9}
	If hypotheses $H_1$ hold, then problem (\ref{eq1}) admits a nontrivial solution $u_0\in C^1_0(\overline{\Omega})$.
\end{theorem}
\begin{proof}
	As we have already mentioned we can use two approaches.
	
	In the first we use Proposition \ref{prop6} and have that
	$$C_m(\varphi,\infty)\neq 0.$$
	
	So, there exists $u_0\in W^{1,p}_0(\Omega)$ such that
	\begin{equation}\label{eq62}
		u_0\in K_{\varphi}\ \mbox{and}\ C_m(\varphi,u_0)\neq 0.
	\end{equation}
	
	On the other hand, from Proposition \ref{prop8} we have
	\begin{equation}\label{eq63}
		C_k(\varphi,0)=\delta_{k,d_l}\ZZ\ \mbox{for all}\ k\in\NN_0.
	\end{equation}
	
	Recalling that $d_l\neq m$ (see hypothesis $H_1(iv)$) and comparing (\ref{eq62}) and (\ref{eq63}), we see that
	$$u_0\neq 0.$$
	
	In the second approach, we use Proposition \ref{prop7}. According to that result, we have
	$$C_m(\varphi_0,\infty)\neq 0.$$
	
	So, we can find $u_0\in W^{1,p}_0(\Omega)$ such that
	\begin{equation}\label{eq64}
		u_0\in K_{\varphi_0}\ \mbox{and}\ C_m(\varphi_0,u_0)\neq 0.
	\end{equation}
	
	Note that $\left.\varphi_0\right|_{\bar{B}_r}=\varphi|_{\bar{B}_r}$ (see Proposition \ref{prop7}). So
	\begin{eqnarray}\label{eq65}
		&&C_k(\varphi_0,0)=C_k(\varphi,0)\ \mbox{for all}\ k\in\NN_0,\nonumber\\
		&\Rightarrow&C_k(\varphi_0,0)=\delta_{k,d_l}\ZZ\ \mbox{for all}\ k\in\NN_0\ (\mbox{see Proposition \ref{prop8}}).
	\end{eqnarray}
	
	Again, since $d_l\neq m$, from (\ref{eq64}) and (\ref{eq65}) it follows that
	$$u_0\neq 0\ \mbox{and}\ u_0\in K_{\varphi}\ \mbox{(see Proposition \ref{prop7})}.$$
	
	So, with both approaches we produced a nontrivial critical point $u_0$ of the functional $\varphi$. Then $u_0$ is a nontrivial solution of (\ref{eq1}). Invoking Theorem 7.1, p. 286, of Ladyzhenskaya and Uraltseva \cite{17}, we have $u_0\in L^{\infty}(\Omega)$. So, we apply Theorem 1 of Lieberman \cite{19} and conclude that $u_0\in C^1(\overline{\Omega})$.
\end{proof}

\section{Multiple Nontrivial Solutions}

In this section we strengthen the conclusions on the reaction term $f(z,x)$ and prove a multiplicity theorem. More precisely, the new conditions on $f(z,x)$ are the following:

\smallskip
$H_2:$ $f:\Omega\times\RR\rightarrow\RR$ is a Carath\'eodory function such that
\begin{itemize}
	\item[(i)] for every $\rho>0$, there exists $a_{\rho}\in L^{\infty}(\Omega)_+$ such that
	$$|f(z,x)|\leq a_{\rho}(z)\ \mbox{for almost all}\ z\in\Omega,\ \mbox{all}\ |x|\leq\rho;$$
	\item[(ii)] there exists an integer $m\geq 1$ such that
	$$\lim\limits_{x\rightarrow\pm\infty}\frac{f(z,x)}{|x|^{p-2}x}=\hat{\lambda}_m(p)\ \mbox{uniformly for almost all}\ z\in\Omega;$$
	\item[(iii)] there exists $\tau\in(2,p)$ such that
	$$0<\beta_0\leq\liminf\limits_{x\rightarrow\pm\infty}\frac{f(z,x)x-pF(z,x)}{|x|^{\tau}}\ \mbox{uniformly for almost all}\ z\in\Omega,$$
	where $F(z,x)=\int^x_0f(z,s)ds$;
	\item[(iv)] there exist functions $w_{\pm}\in W^{1,p}(\Omega)\cap C(\overline{\Omega})$ and constants $c_{\pm}\in\RR$ such that
	\begin{eqnarray*}
		&&w_-(z)\leq c_-<0<c_+\leq w_+(z)\ \mbox{for all}\ z\in\overline{\Omega},\\
		&&f(z,w_+(z))\leq 0\leq f(z,w_-(z))\ \mbox{for almost all}\ z\in\Omega,\\
		&&A_p(w_-)+A(w_-)\leq 0\leq A_p(w_+)+A(w_+)\ \mbox{in}\ W^{-1,p'}(\Omega)=W^{1,p}_0(\Omega)^*;
	\end{eqnarray*}
	\item[(v)] there exist an integer $l\geq 1$ with $d_l\neq m$ ($d_l={\rm dim}\,\bar{H}_l$), $\delta>0$ and $\eta\in L^{\infty}(\Omega)$ such that
	\begin{eqnarray*}
		&&\hat{\lambda}_l(2)\leq\eta(z)\ \mbox{for almost all}\ z\in\Omega,\eta\not\equiv\hat{\lambda}_l(2),\\
		&&\eta(z)x^2\leq f(z,x)x\leq\hat{\lambda}_{l+1}(2)x^2\ \mbox{for almost all}\ z\in\Omega,\ \mbox{all}\ |x|\leq\delta
	\end{eqnarray*}
	and for $x\neq 0$ the second inequality is strict on a set of positive Lebesgue measure;
	\item[(vi)] for every $\rho>0$, there exists $\hat{\xi}_{\rho}>0$ such that for almost all $z\in\Omega$ the function
	$$z\mapsto f(z,x)+\hat{\xi}_{\rho}|x|^{p-2}x$$
	is nondecreasing on $[-\rho,\rho]$.
\end{itemize}
\begin{remark}
We see that in comparison to the hypotheses
$H_1$, we have added
	 hypotheses $H_2(iv),\, (vi)$. So, the problem remains resonant at both $\pm\infty$ and at zero. Hypothesis $H_1(iv)$ is satisfied if for example, we can find $c_-<0<c_+$ such that
	$$f(z,c_+)\leq 0\leq f(z,c_-)\ \mbox{for almost all}\ z\in\Omega\,.$$
	
	Therefore this hypothesis implies that near zero $f(z,\cdot)$ exhibits an oscillatory behavior.
\end{remark}

First, we produce two constant sign solutions.
\begin{prop}\label{prop10}
	If hypotheses $H_2(i),(iv),(v),(vi)$ hold, then problem (\ref{eq1}) admits two nontrivial smooth solutions of constant sign
	\begin{eqnarray*}
		&&u_0\in {\rm int}\,C_+\ \mbox{with}\ u_0(z)<w_+(z)\ \mbox{for all}\ z\in\overline{\Omega},\\
		&&v_0\in -{\rm int}\,C_+\ \mbox{with}\ w_-(z)<v_0(z)\ \mbox{for all}\ z\in\overline{\Omega}.
	\end{eqnarray*}
\end{prop}
\begin{proof}
	First, we produce the positive solution.
	
	We introduce the following Carath\'eodory function
	\begin{equation}\label{eq66}
		\hat{f}_+(z,x)=\left\{\begin{array}{ll}
			0&\mbox{if}\ x<0\\
			f(z,x)&\mbox{if}\ 0\leq x\leq w_+(z)\\
			f(z,w_+(z))&\mbox{if}\ w_+(z)<x.
		\end{array}\right.
	\end{equation}
	
	We set $\hat{F}_+(z,x)=\int^x_0\hat{f}_+(z,s)ds$ and consider the $C^1$-functional $\hat{\varphi}_+:W^{1,p}_0(\Omega)\rightarrow\RR$ defined by
	$$\hat{\varphi}_+(u)=\frac{1}{p}||Du||^p_p+\frac{1}{2}||Du||^2_2-\int_{\Omega}\hat{F}_+(z,u)dz\ \mbox{for all}\ u\in W^{1,p}_0(\Omega).$$
	
	From (\ref{eq66}) it is clear that $\hat{\varphi}_+$ is coercive. Also, using the Sobolev embedding theorem, we see that $\hat{\varphi}_+$ is sequentially weakly lower semicontinuous. So, by the Weierstrass theorem, we can find $u_0\in W^{1,p}_0(\Omega)$ such that
	\begin{equation}\label{eq67}
		\hat{\varphi}_+(u_0)=\inf[\hat{\varphi}_+(u):u\in W^{1,p}_0(\Omega)].
	\end{equation}
	
	From (\ref{eq67}) we have
	\begin{eqnarray}\label{eq68}
		&&\hat{\varphi}'_+(u_0)=0,\nonumber\\
		&\Rightarrow&\left\langle A_p(u_0),h\right\rangle+\left\langle A(u_0),h\right\rangle=\int_{\Omega}\hat{f}_+(z,u_0)hdz\ \mbox{for all}\ h\in W^{1,p}_0(\Omega).
	\end{eqnarray}
	
	In (\ref{eq68})  we first choose $h=-u^-_0\in W^{1,p}_0(\Omega)$. Then
	\begin{eqnarray*}
		&&||Du^-_0||^p_p+||Du^-_0||^2_2=0\ (\mbox{see (\ref{eq66})}),\\
		&\Rightarrow&u_0\geq 0.
	\end{eqnarray*}
	
	Also, in (\ref{eq68}) we choose $h=(u_0-w_+)^+\in W^{1,p}_0(\Omega)$ (see hypothesis $H_2(iv)$). Then
	\begin{eqnarray*}
		&&\int_{\Omega}|Du_0|^{p-2}(Du_0,D(u_0-w_+)^+)_{\RR^N}dz+\int_{\Omega}(Du_0,D(u_0-w_+)^+)_{\RR^N}dz\\
		&&=\int_{\Omega}f(z,w_+)(u_0-w_+)^+dz\ (\mbox{see (\ref{eq66})})\\
		&&\leq\int_{\Omega}|Dw_+|^{p-2}(Dw_+,D(u_0-w_+)^+)_{\RR^N}dz+\int_{\Omega}(Dw_+,D(u_0-w_+)^+)_{\RR^N}dz\\
		&&(\mbox{see hypothesis}\ H_2(iv)),\\
		&\Rightarrow&\int_{\Omega}(|Du_0|^{p-2}Du_0-|Dw_+|^{p-2}Dw_+,D(u_0-w^+)^+)_{\RR^N}dz+||D(u_0-w_+)^+||^2_2\leq 0,\\
		&\Rightarrow&u_0\leq w_+.
	\end{eqnarray*}
	
	So, we have proved that
	$$u_0\in[0,w_+]=\{y\in W^{1,p}_0(\Omega):0\leq y(z)\leq w_+(z)\ \mbox{for almost all}\ z\in\Omega\}.$$
	
	Then on account of (\ref{eq66}), equation (\ref{eq68}) becomes
	\begin{eqnarray}\label{eq69}
		&&\left\langle A_p(u_0),h\right\rangle+\left\langle A(u_0),h\right\rangle=\int_{\Omega}f(z,u_0)hdz\ \mbox{for all}\ h\in W^{1,p}_0(\Omega),\nonumber\\
		&\Rightarrow&-\Delta_pu_0(z)-\Delta u_0(z)=f(z,u_0(z))\ \mbox{for almost all}\ z\in\Omega,\ u_0|_{\partial\Omega}=0,\\
		&\Rightarrow&u_0\in C_+\ (\mbox{by the nonlinear regularity theory, see Lieberman \cite{19}})\nonumber.
	\end{eqnarray}
	
	Since $p>2$, given $\epsilon>0$, we can find $\delta_0\in(0,\min\{\delta,C_+\})$ ($\delta>0$ as in hypothesis $H_2(v)$) such that
	\begin{equation}\label{eq70}
		\frac{1}{p}|y|^p\leq\frac{\epsilon}{2}|y|^2\ \mbox{for all}\ y\in\RR^N\ \mbox{with}\ |y|\leq\delta_0.
	\end{equation}
	
	Recall that $\hat{u}_1(2)\in {\rm int}\,C_+$. So, we can find  small $t\in(0,1)$ such that
	$$||t\hat{u}_1(2)||_{C^1_0(\overline{\Omega})}\leq\delta_0.$$
	
	We have
	\begin{eqnarray*}
		 \hat{\varphi}_+(t\hat{u}_1(2))&\leq&\frac{\epsilon+1}{2}t^2||D\hat{u}_1(2)||^2_2-\frac{1}{2}t^2\int_{\Omega}\eta(z)\hat{u}_1(2)^2dz\\
		&&(\mbox{see (\ref{eq70}), (\ref{eq66}) and hypothesis}\ H_2(v),\ \mbox{since}\ \delta_0\leq\delta)\\
		 &=&t^2\left[\frac{\epsilon}{2}\hat{\lambda}_1(2)
||\hat{u}_1(2)||^2_2
-\frac{1}{2}\int_{\Omega}(\eta(z)-\hat{\lambda}_1(2))\hat{u}_1(2)^2dz)\right]\\
		&<&0\ \mbox{choosing}\ \epsilon>0\ \mbox{small (see Lemma \ref{lem2}(b))},\\
		\Rightarrow\hat{\varphi}_+(u_0)<0&=&\hat{\varphi}_+(0)\ (\mbox{see (\ref{eq67})}),\\
		\Rightarrow u_0\neq 0.&&
	\end{eqnarray*}
	
	Let $\rho=||u_0||_{\infty}$ and let $\hat{\xi}_{\rho}>0$ be as postulated by hypothesis $H_1'(vi)$. Then by (\ref{eq69}), we have
	\begin{equation}\label{eq71}
		\Delta_pu_0(z)+\Delta u_0(z)\leq\hat{\xi}_{\rho}u_0(z)^{p-1}\ \mbox{for almost all}\ z\in\Omega\,.
	\end{equation}
	
	Let $V(y)=|y|^{p-2}y+y$ for all $y\in\RR^N$. Evidently,
	$${\rm div}\, (V(Du))=\Delta_pu+\Delta u\ \mbox{for all}\ u\in W^{1,p}_0(\Omega).$$
	
	We have $V\in C^1(\RR^N,\RR^N)$ and
	\begin{eqnarray}\label{eq72}
		&&\nabla V(y)=|y|^{p-2}\left[I+(p-2)\frac{y\oplus y}{|y|^2}\right]+I,\nonumber\\
		&\Rightarrow&(\nabla V(y)\xi,\xi)_{\RR^N}\geq|\xi|^2\ \mbox{for all}\ y\in\RR^N,\ \mbox{all}\ \xi\in\RR^N.
	\end{eqnarray}
	
	Then (\ref{eq72}), (\ref{eq71}) and the tangency principle of Pucci and Serrin \cite[Theorem 2.5.2, p. 35]{30}, imply that
	$$0<u_0(z)\ \mbox{for all}\ z\in\Omega\,.$$
	
	Next, using the boundary point lemma (see Pucci and Serrin \cite[Theorem 5.5.1, p. 120]{30}), we obtain
	\begin{equation}\label{eq73}
		u_0\in {\rm int}\,C_+.
	\end{equation}
	
	Also, hypothesis $H_2(iv)$ implies
	\begin{equation}\label{eq74}
		A_p(u_0)+A(u_0)-N_f(u_0)=0\leq A_p(w_+)+A(w_+)-N_f(w_+)\ \mbox{in}\ W^{-1,p'}(\Omega).
	\end{equation}
	
	So, once more (\ref{eq72}), (\ref{eq74}) and the tangency principle of Pucci and Serrin \cite[p. 35]{30}, imply that
	\begin{eqnarray*}
		&&u_0(z)<w_+(z)\ \mbox{for all}\ z\in\Omega,\\
		&\Rightarrow&u_0(z)<w_+(z)\ \mbox{for all}\ z\in\overline{\Omega}\ (\mbox{see hypothesis}\ H_2(iv)).
	\end{eqnarray*}
	
	Similarly, to produce the negative solution, we introduce the Carath\'eodory function
	\begin{equation}\label{eq75}
		\hat{f}_-(z,x)=\left\{\begin{array}{ll}
			f(z,w_-(z))&\mbox{if}\ x<w_-(z)\\
			f(z,x)&\mbox{if}\ w_-(z)\leq x\leq 0\\
			0&\mbox{if}\ 0<x.
		\end{array}\right.
	\end{equation}
	
	We set $\hat{F}_-(z,x)=\int^x_0\hat{f}_-(z,s)ds$ and consider the $C^1$-functional $\hat{\varphi}_-:W^{1,p}_0(\Omega)\rightarrow\RR$ defined by
	$$\hat{\varphi}_-(u)=\frac{1}{p}||Du||^p_p+\frac{1}{2}||Du||^2_2-\int_{\Omega}\hat{F}_-(z,u)dz\ \mbox{for all}\ u\in W^{1,p}_0(\Omega).$$
	
	Working with $\hat{\varphi}_-$ and using (\ref{eq73}), we produce $v_0\in W^{1,p}_0(\Omega)$ solution of (\ref{eq1}) such that
	$$v_0\in-{\rm int}\,C_+,\ w_-(z)<v_0(z)\ \mbox{for all}\ z\in\overline{\Omega}.$$
\end{proof}

In fact, we can show that we have extremal constant sign solutions in the order intervals $[0,w_+]$ and $[w_-,0]$, that is, we show that there is a smallest positive solution $u_*\in {\rm int}\,C_+$ in $[0,w_+]$ and a biggest negative solution $v_*\in-{\rm int}\,C_+$ in $[w_-,0]$.
\begin{prop}\label{prop11}
	If hypotheses $H_2(i),(iv),(v),(vi)$ hold, then problem (\ref{eq1}) admits a smallest positive solution $u_*\in {\rm int}\,C_+$ in $[0,w_+]$ and a biggest negative solution $v_*\in-{\rm int}\,C_+$ in $[w_-,0]$.
\end{prop}
\begin{proof}
	First we produce the smallest positive solution in $[0,w_+]$.
	
	Let $\hat{S}_+$ be the set of positive solutions of problem (\ref{eq1}) in the order interval $[0,w_+]$. From Proposition \ref{prop10} and its proof, we have
	$$\hat{S}_+\neq\emptyset\ \mbox{and}\ \hat{S}_+\subseteq[0,w_+]\cap {\rm int}\,C_+.$$
	
	Invoking Lemma 3.10, p. 178, of Hu and Papageorgiou \cite{16}, we infer that we can find\\ $\{u_n\}_{n\geq 1}\subseteq\hat{S}_+$ such that
	$$\inf\hat{S}_+=\inf\limits_{n\geq 1}u_n.$$
	
	We have
	\begin{eqnarray}\label{eq76}
		&&A_p(u_n)+A(u_n)=N_f(u_n),\ 0\leq u_n\leq w_+\ \mbox{for all}\ n\in\NN,\\
		&\Rightarrow&\{u_n\}_{n\geq 1}\subseteq W^{1,p}_0(\Omega)\ \mbox{is bounded}\nonumber.
	\end{eqnarray}
	
	So, we may assume that
	\begin{equation}\label{eq77}
		u_n\stackrel{w}{\rightarrow}u_*\ \mbox{in}\ W^{1,p}_0(\Omega)\ \mbox{and}\ u_n\rightarrow u_*\ \mbox{in}\ L^p(\Omega)\ \mbox{as}\ n\rightarrow\infty.
	\end{equation}
	
	On (\ref{eq76}) we act with $u_n-u_*\in W^{1,p}_0(\Omega)$, pass to the limit as $n\rightarrow\infty$ and use (\ref{eq77}). Then
	\begin{eqnarray}\label{eq78}
		&&\lim\limits_{n\rightarrow\infty}\left[\left\langle A_p(u_n),u_n-u_*\right\rangle+\left\langle A(u_n),u_n-u_*\right\rangle\right]=0,\nonumber\\
		&\Rightarrow&\limsup\limits_{n\rightarrow\infty}\left[\left\langle A_p(u_n),u_n-u_*\right\rangle+\left\langle A(u_*),u_n-u_*\right\rangle\right]\leq 0\ (\mbox{recall $A$ is monotone})\nonumber\\
		&\Rightarrow&\limsup\limits_{n\rightarrow\infty}\left\langle A_p(u_n),u_n-u_*\right\rangle\leq 0\ (\mbox{see (\ref{eq77})})\nonumber\\
		&\Rightarrow&u_n \rightarrow u_*\ \mbox{in}\ W^{1,p}_0(\Omega)\ (\mbox{see Proposition \ref{prop3}}).
	\end{eqnarray}
	
	Passing to the limit as $n\rightarrow\infty$ in (\ref{eq76}) and using (\ref{eq78}), we obtain
	\begin{eqnarray}\label{eq79}
		&&A_p(u_*)+A(u_*)=N_f(u_*),\ 0\leq u_*\leq w_+,\\
		&\Rightarrow&-\Delta_pu_*(z)-\Delta u_*(z)=f(z,u_*(z))\ \mbox{for almost all}\ z\in\Omega,\ u_*|_{\partial\Omega}=0,\ 0\leq u_*\leq w_+.\nonumber
	\end{eqnarray}
	
	Then $u_*\in C_+$ (by the nonlinear regularity theory, see Lieberman \cite{19}) is a nonnegative solution of (\ref{eq1}). If we can show that $u_*\neq 0$, then $u_*\in \hat{S}_+$ and $u_*=\inf\hat{S}_+$.
	
	To this end, we proceed as follows. Hypotheses $H_2(i),(v)$ imply that we can find $c_{13}>0$ such that
	\begin{equation}\label{eq80}
		f(z,x)\geq \eta(z)x-c_{13}x^{p-1}\ \mbox{for almost all}\ z\in\Omega,\ \mbox{all}\ 0\leq x\leq w_+(z).
	\end{equation}
	
	Let $g:\Omega\times\RR\rightarrow\RR$ be the Carath\'eodory function defined by
	\begin{equation}\label{eq81}
		g(z,x)=\left\{\begin{array}{ll}
			0&\mbox{if}\ x<0\\
			\eta(z)x-c_{13}x^{p-1}&\mbox{if}\ 0\leq x\leq w_+(z)\\
			\eta(z)w_+(z)-c_{13}w_+(z)^{p-1}&\mbox{if}\ w_+(z)<x.
		\end{array}\right.
	\end{equation}
	
	We consider the auxiliary Dirichlet problem
	\begin{equation}\label{eq82}
		-\Delta_pu(z)-\Delta u(z)=g(z,u(z))\ \mbox{in}\ \Omega,\ u|_{\partial\Omega}=0.
	\end{equation}
	
	We claim that this problem has a unique solution $\bar{u}\in {\rm int}\,C_+$. First, we show the existence of a nontrivial solution. So, let $\psi_+:W^{1,p}_0(\Omega)\rightarrow\RR$ be the energy (Euler) functional for problem (\ref{eq82}) defined by
	$$\psi_+(u)=\frac{1}{p}||Du||^p_p+\frac{1}{2}||Du||^2_2-\int_{\Omega}G(z,u)dz\ \mbox{for all}\ u\in W^{1,p}_0(\Omega),$$
	where $G(z,x)=\int^x_0g(z,s)ds$. Evidently, $\psi_+$ is coercive (see (\ref{eq81})) and sequentially weakly lower semicontinuous. So, we can find $\bar{u}\in W^{1,p}_0(\Omega)$ such that
	$$\psi_+(\bar{u})=\inf[\psi_+(u):u\in W^{1,p}_0(\Omega)].$$
	
	As in the proof of Proposition \ref{prop10}, using hypothesis $H_2(v)$, we have
	\begin{eqnarray*}
		&&\psi_+(\bar{u})<0=\psi_+(0)\ \mbox{and}\ \bar{u}\in[0,w_+]\ (\mbox{see (\ref{eq81})}),\\
		&\Rightarrow&\bar{u}\in K_{\psi_+}\subseteq[0,w_+]\cap {\rm int}\,C_+.
	\end{eqnarray*}
	
	Next, we show that this solution is unique. For this purpose, we consider the integral functional $j:L^1(\Omega)\rightarrow\overline{\RR}=\RR\cup\{+\infty\}$ defined by
	$$j(u)=\left\{\begin{array}{ll}
		\frac{1}{p}||Du^{1/2}||^p_p+\frac{1}{2}||Du^{1/2}||^2_2&\mbox{if}\ u\geq 0,\ u^{1/2}\in W^{1,p}_0(\Omega)\\
		+\infty&\mbox{otherwise}.
	\end{array}\right.$$
	
	By Lemma 4 of Benguria, Brezis and Lieb \cite{6} and Lemma 1 of Diaz and Saa \cite{12}, we have that $j(\cdot)$ is convex.
	
	Suppose that $\bar{y}\in W^{1,p}_0(\Omega)$ is another nontrivial solution of (\ref{eq82}). Then again we have $\bar{y}\in[0,w_+]\cap {\rm int}\,C_+$. Let ${\rm dom}\,j=\{u\in L^1(\Omega):j(u)<+\infty\}$ (the effective domain of $j$). For every $h\in C^1_0(\overline{\Omega})$, we have
	$$\bar{u}^2+th\in {\rm dom}\, j\ \mbox{and}\ \bar{y}^2+th\in {\rm dom}\, j\ \mbox{for}\ |t|\leq 1\ \mbox{small}.$$
	
	Then we can easily see that $j(\cdot)$ is G\^ateaux differentiable at $\bar{u}^2$ and at $\bar{y}^2$ in the direction $h$. Moreover, using the nonlinear Green's identity (see, for example, Gasinski and Papageorgiou \cite[p. 211]{15}), we have
	\begin{eqnarray*}
		 &&j'(\bar{u}^2)(h)=\frac{1}{2}\int_{\Omega}\frac{-\Delta_p\bar{u}-\Delta\bar{u}}{\bar{u}}hdz=\frac{1}{2}\int_{\Omega}[\eta(z)-c_{13}\bar{u}^{p-2}]hdz\\
		 &&j'(\bar{y}^2)(h)=\frac{1}{2}\int_{\Omega}\frac{-\Delta_p\bar{y}-\Delta\bar{y}}{\bar{y}}hdz=\frac{1}{2}\int_{\Omega}[\eta(z)-c_{13}\bar{y}^{p-2}]hdz\ (\mbox{see (\ref{eq82}), (\ref{eq81})}).
	\end{eqnarray*}
	
	The convexity of $j(\cdot)$, implies monotonicity of $j'(\cdot)$. Hence
	\begin{eqnarray*}
		&&0\leq\int_{\Omega}[\bar{y}^{p-2}-\bar{u}^{p-2}](\bar{u}^2-\bar{y}^2)dz,\\
		&\Rightarrow&\bar{u}=\bar{y}.
	\end{eqnarray*}
	
	This proves the uniqueness of the nontrivial solution $\bar{u}\in[0,w_+]\cap {\rm int}\,C_+$ of the auxiliary problem (\ref{eq82}).
	\begin{claim}
		$\bar{u}\leq u$ for all $u\in\hat{S}_+$.
	\end{claim}
	
	Let $u\in\hat{S}_+$ and consider the Carath\'eodory function $k:\Omega\times\RR\rightarrow\RR$ defined by
	\begin{equation}\label{eq83}
		k(z,x)=\left\{\begin{array}{ll}
			0&\mbox{if}\ x<0\\
			\eta(z)x-c_{13}x^{p-1}&\mbox{if}\ 0\leq x\leq u(z)\\
			\eta(z)u(z)-c_{13}u(z)^{p-1}&\mbox{if}\ u(z)<x.
		\end{array}\right.
	\end{equation}
	
	We set $K(z,x)=\int^x_0k(z,s)ds$ and consider the $C^1$-functional $\hat{\psi}_+:W^{1,p}_0(\Omega)\rightarrow\RR$ defined by
	$$\hat{\psi}_+(u)=\frac{1}{p}||Du||^p_p+\frac{1}{2}||Du||^2_2-\int_{\Omega}K(z,u)dz\ \mbox{for all}\ u\in W^{1,p}_0(\Omega).$$
	
	Again, $\hat{\psi}_+$ is coercive (see (\ref{eq83})) and sequentially weakly lower semicontinuous. So, we can find $\tilde{u}\in W^{1,p}_0(\Omega)$ such that
	\begin{equation}\label{eq84}
		\hat{\psi}_+(\tilde{u})=\inf[\hat{\psi}_+(u):u\in W^{1,p}_0(\Omega)].
	\end{equation}
	
	Let $t\in(0,1)$ be small such that $t\hat{u}_1(2)\leq u$ (see Proposition 2.1 of Marano and Papageorgiou \cite{20} and recall that $u\in {\rm int}\,C_+$). Then by taking $t\in(0,1)$ even smaller if necessary and using hypothesis $H_2(v)$, we have
	\begin{eqnarray*}
		&&\hat{\psi}_+(t\hat{u}_1(2))<0,\\
		&\Rightarrow&\hat{\psi}_+(\tilde{u})<\hat{\psi}_+(0)=0,\ \mbox{hence}\ \tilde{u}\neq 0.
	\end{eqnarray*}
	
	Using (\ref{eq80}) and the fact that $u\in\hat{S}_+$, we show that $K_{\hat{\psi}_+}\subseteq[0,u]$. From (\ref{eq84}) we have
	\begin{eqnarray*}
		&&\tilde{u}\in K_{\hat{\psi}_+}\backslash\{0\}\subseteq[0,u]\backslash\{0\}\\
		&\Rightarrow&\tilde{u}=\bar{u}\ (\mbox{see (\ref{eq83}) and recall that}\ \bar{u}\ \mbox{is the unique solution of (\ref{eq82})})\\
		&\Rightarrow&\bar{u}\leq u\ \mbox{for all}\ u\in \hat{S}_+.
	\end{eqnarray*}
	
	This proves the claim.
	
	On account of Claim 4 we have $\bar{u}\leq u_*$ and so
	$$u_*\in\hat{S}_+,\ u_*=\inf\hat{S}_+\,.$$
	
	Similarly, if $\hat{S}_-$ is the set of negative solutions of (\ref{eq1}) in $[w_-,0]$, then
	$$\hat{S}_-\neq\emptyset\ \mbox{and}\ \hat{S}_-\subseteq[w_-,0]\cap(-{\rm int}\,C_+)$$
	(see Proposition \ref{prop10} and its proof). Reasoning as above, we show that there exists $v_*\in[w_-,0]\cap(-{\rm int}\,C_+)$ the biggest negative solution of (\ref{eq1}) in $[w_-,0]$.
\end{proof}

Using these extremal constant sign solutions of (\ref{eq1}), we can generate a nodal (that is, sign changing) solution. To do this, we need a slightly stronger condition on $f(z,\cdot)$ near zero (see hypothesis $H_2(v)$). The new hypotheses on the reaction $f(z,x)$ are the following:

\smallskip
$H_3:$ The conditions on the Carath\'eodory function $f:\Omega\times\RR\rightarrow\RR$ are the same as in $H_2$ the only difference being that in $H_3(v)$ we have $l\geq 2$.
\begin{prop}\label{prop12}
	If hypotheses $H_3(i),(iv),(v),(vi)$ hold, then problem (\ref{eq1}) admits a nodal solution $y_0\in[v_*,u_*]\cap C^1_0(\bar{\Omega})$.
\end{prop}
\begin{proof}
	Let $u_*\in {\rm int}\,C_+$ and $v_*\in-{\rm int}\,C_+$ be the two extremal constant sign solutions of (\ref{eq1}) produced in Proposition \ref{prop11}. Let $e:\Omega\times\RR\rightarrow\RR$ be the Carath\'eodory function defined by
	\begin{equation}\label{eq85}
		e(z,x)=\left\{\begin{array}{ll}
			f(z,v_*(z))&\mbox{if}\ x<v_*(z)\\
			f(z,x)&\mbox{if}\ v_*(z)\leq x\leq u_*(z)\\
			f(z,u_*(z))&\mbox{if}\ u_*(z)<x.
		\end{array}\right.
	\end{equation}
	
	We set $E(z,x)=\int^x_0e(z,s)ds$ and consider the $C^1$-functional $\tau:W^{1,p}_0(\Omega)\rightarrow\RR$ defined by
	$$\tau(u)=\frac{1}{p}||Du||^p_p+\frac{1}{2}||Du||^2_2-\int_{\Omega}E(z,u)dz\ \mbox{for all}\ u\in W^{1,p}_0(\Omega).$$
	
	Also, we consider the positive and negative truncations of $e(z,\cdot)$, namely the Carath\'eodory functions
	$$e_{\pm}(z,x)=e(z,\pm x^{\pm}).$$
	
	We set $E_{\pm}(z,x)=\int^x_0e_{\pm}(z,s)ds$ and consider the $C^1$-functionals $\tau_{\pm}:W^{1,p}_0(\Omega)\rightarrow\RR$ defined by
	$$\tau_{\pm}(u)=\frac{1}{p}||Du||^p_p+\frac{1}{2}||Du||^2_2-\int_{\Omega}E_{\pm}(z,u)dz\ \mbox{for all}\ u\in W^{1,p}_0(\Omega).$$
	
	As before (see the proof of Proposition \ref{prop10}), using (\ref{eq85}), we can show that
	$$K_{\tau}\subseteq[v_*,u_*],\ K_{\tau_+}\subseteq[0,u_*],\ K_{\tau_-}\subseteq[v_*,0].$$
	
	The extremality of $u_*\in {\rm int}\,C_+$ and $v_*\in -{\rm int}\,C_+$ implies that
	\begin{equation}\label{eq86}
		K_{\tau}\subseteq[v_*,u_*],\ K_{\tau_+}=\{0,u_*\},\ K_{\tau_-}=\{0,v_*\}.
	\end{equation}
	\begin{claim}\label{cl1}
		$u_*\in {\rm int}\,C_+$ and $v_*\in-{\rm int}\,C_+$ are local minimizers of $\tau$.
	\end{claim}
	
	The functional $\tau_+$ is coercive (\ref{eq85}) and sequentially weakly lower semicontinuous. So, we can find $\hat{u}_*\in W^{1,p}_0(\Omega)$ such that
	$$\tau_+(\hat{u}_*)=\inf[\tau_+(u):u\in W^{1,p}_0(\Omega)].$$
	
	As in the proof of Proposition \ref{prop11} (see the part of the proof immediately after (\ref{eq84})), we have $\tau_+(\hat{u}_*)<0=\tau_+(0)$, hence $\hat{u}_*\neq 0$. Since $\hat{u}_*\in K_{\tau_+}=\{0,u_*\}$, it follows that $\hat{u}_*=u_*\in {\rm int}\,C_+$ (see (\ref{eq86})). Note that
	\begin{eqnarray*}
		&&\tau|_{C_+}=\tau_+|_{C_+},\\
		&\Rightarrow&u_*\in {\rm int}\,C_+\ \mbox{is a local}\ C^1_0(\overline{\Omega})-\mbox{minimizer of}\ \tau,\\
		&\Rightarrow&u_*\in {\rm int}\,C_+\ \mbox{is a local}\ W^{1,p}_0(\Omega)-\mbox{minimizer of}\ \tau\ (\mbox{see Proposition \ref{prop4}}).
	\end{eqnarray*}
	
	Similarly for $v_*\in-{\rm int}\,C_+$, using this time the functional $\tau_-$.
	
	This proves Claim \ref{cl1}.
	
	We may assume that
	$$\tau(v_*)\leq\tau(u_*).$$
	
	The reasoning is similar if the opposite inequality holds. Also, we may assume that $K_{\tau}$ is finite. Indeed, if $K_{\tau}$ is infinite, then on account of (\ref{eq86}) we see that we already have an infinity of nodal solutions, which belong to $C^1_0(\overline{\Omega})$ (nonlinear regularity theory). Then Claim \ref{cl1} implies that we can find $\rho\in(0,1)$ small such that
	\begin{equation}\label{eq87}
		\tau(v_*)\leq\tau(u_*)<\inf[\tau(u):||u-u_*||=\rho]=m_{\rho},\ ||v_*-u_*||>\rho
	\end{equation}
	(see Aizicovici, Papageorgiou and Staicu \cite{1}, proof of Proposition 29). The functional $\tau(\cdot)$ is coercive (see (\ref{eq85})) and so
	\begin{equation}\label{eq88}
		\tau(\cdot)\ \mbox{satisfies the C-condition}
	\end{equation}
	(see Papageorgiou and Winkert \cite{28}).
	
	Because of (\ref{eq87}), (\ref{eq88}), we see that we can apply Theorem \ref{th1} (the mountain pass theorem). So, we can find $y_0\in W^{1,p}_0(\Omega)$ such that
	\begin{equation}\label{eq89}
		y_0\in K_{\tau}\ \mbox{and}\ m_{\rho}\leq\tau(y_0).
	\end{equation}
	
	From (\ref{eq86}), (\ref{eq87}), (\ref{eq89}) and the nonlinear regularity theory (see \cite{19}), we infer that
	$$y_0\in[v_*,u_*]\cap C^1_0(\overline{\Omega}),\ y_0\notin\{v_*,u_*\}.$$
	
	Also, from Corollary 6.81, p. 168, of Motreanu, Motreanu and Papageorgiou \cite{21}, we have
	\begin{equation}\label{eq90}
		C_1(\tau,y_0)\neq 0.
	\end{equation}
	
	Let $\hat{f}:\Omega\times\RR\rightarrow\RR$ be the Carath\'eodory function defined by
	\begin{equation}\label{eq91}
		\hat{f}(z,x)=\left\{\begin{array}{ll}
			f(z,w_-(z))&\mbox{if}\ x<w_-(z)\\
			f(z,x)&\mbox{if}\ w_-(z)\leq x\leq w_+(z)\\
			f(z,w_+(z))&\mbox{if}\ w_+(z)<x.
		\end{array}\right.
	\end{equation}
	
	We set $\hat{F}(t,x)=\int^x_0\hat{f}(z,s)ds$ and consider the $C^1$-functional $\hat{\varphi}:W^{1,p}_0(\Omega)\rightarrow\RR$ defined by
	$$\hat{\varphi}(u)=\frac{1}{p}||Du||^p_p+\frac{1}{2}||Du||^2_2-\int_{\Omega}\hat{F}(z,u)dz\ \mbox{for all}\ u\in W^{1,p}_0(\Omega).$$
	
	From Proposition \ref{prop8}, we know that
	\begin{equation}\label{eq92}
		C_k(\hat{\varphi},0)=\delta_{k,d_l}\ZZ\ \mbox{for all}\ k\in\NN_0\ (\mbox{recall}\ d_l={\rm dim}\,\bar{H}_l)
	\end{equation}
	\begin{claim}\label{cl2}
		$C_k(\tau,0)=\delta_{k,d_l}\ZZ$ for all $k\in\NN_0$.
	\end{claim}
	
	We consider the homotopy $h(t,u)$ defined by
	$$h(t,u)=(1-t)\hat{\varphi}(u)+t\tau(u)\ \mbox{for all}\ (t,u)\in[0,1]\times W^{1,p}_0(\Omega).$$
	
	Suppose we can find $\{t_n\}_{n\geq 1}\subseteq[0,1]$ and $\{u_n\}_{n\geq 1}\subseteq W^{1,p}_0(\Omega)$ such that
	\begin{equation}\label{eq93}
		t_n\rightarrow t,\ u_n\rightarrow 0\ \mbox{in}\ W^{1,p}_0(\Omega),\ h'_u(t_n,u_n)=0\ \mbox{for all}\ n\in\NN.
	\end{equation}
	
	From the equality in (\ref{eq93}) we have
	\begin{eqnarray}\label{eq94}
		&&A_p(u_n)+A(u_n)=(1-t_n)N_{\hat{f}}(u_n)+t_nN_{\tau}(u_n)\ \mbox{for all}\ n\in\NN,\nonumber\\
		&\Rightarrow&-\Delta_pu_n(z)-\Delta u_n(z)=(1-t_n)\hat{f}(z,u_n(z))+t_ne(z,u_n(z))\\
		&&\mbox{for almost all}\ z\in\Omega,\ u_n|_{\partial\Omega}=0.\nonumber
	\end{eqnarray}
	
	By (\ref{eq93}), (\ref{eq94}) and Theorem 7.1, p. 286, of Ladyzhenskaya and Uraltseva \cite{17} (see also Corollary 8.7, p. 208, of Motreanu, Motreanu and Papageorgiou \cite{21}), we can find $c_{14}>0$ such that
	\begin{equation}\label{eq95}
		||u_n||_{\infty}\leq c_{14}\ \mbox{for all}\ n\in\NN\,.
	\end{equation}
	
	Then from (\ref{eq95}) and Theorem 1 of Lieberman \cite{19}, we infer that there exist $\alpha\in(0,1)$ and $c_{15}>0$ such that
	$$u_n\in C^{1,\alpha}_0(\overline{\Omega}),\ ||u_n||_{C^{1,\alpha}_0(\overline{\Omega})}\leq c_{15}\ \mbox{for all}\ n\in\NN.$$
	
	Since $C^{1,\alpha}_0(\overline{\Omega})$ is embedding compactly in $C^1_0(\overline{\Omega})$, it follows that
	\begin{eqnarray*}
		&&u_n\rightarrow 0\ \mbox{in}\ C^1_0(\overline{\Omega})\ (\mbox{see (\ref{eq93})}),\\
		&\Rightarrow&u_n\in[v_*,u_*]\ \mbox{for all}\ n\geq n_0,\\
		&\Rightarrow&\{u_n\}_{n\geq n_0}\subseteq K_{\tau}\ (\mbox{see (\ref{eq86})}),
	\end{eqnarray*}
	a contradiction to our hypothesis that $K_{\tau}$ is finite.
	
	So, (\ref{eq93}) cannot happen and this shows that $0\in K_{h(t,\cdot)}$ is isolated uniformly in $t\in[0,1]$. Hence Theorem 5.2 of Corvellec and Hantoute \cite{11} (the homotopy invariance of critical groups), implies that
	\begin{eqnarray*}
		&&C_k(h(0,\cdot),0)=C_k(h(1,\cdot),0)\ \mbox{for all}\ k\in\NN_0,\\
		&\Rightarrow&C_k(\hat{\varphi},0)=C_k(\tau,0)\ \mbox{for all}\ k\in\NN_0,\\
		&\Rightarrow&C_k(\tau,0)=\delta_{k,d_l}\ZZ\ \mbox{for all}\ k\in\NN_0\ (\mbox{see (\ref{eq92})}).
	\end{eqnarray*}
	
	This proves Claim \ref{cl2}.
	
	Since $l\geq 2$ (see hypotheses $H_3$), we have $d_l\geq 2$. So, from Claim \ref{cl2} and (\ref{eq90}), it follows that $y_0\neq 0$. Therefore $y_0\in[v_*,u_*]\cap C^1_0(\overline{\Omega})\backslash\{0\}$ is nodal.
\end{proof}

So far we have not used the asymptotic conditions at $\pm\infty$ (that is, hypotheses $H_3(ii),(iii)$). Next, using them we will generate two more nontrivial smooth solutions of constant sign, for a total of five nontrivial smooth solutions all with sign information and ordered.
\begin{theorem}\label{th13}
	If hypotheses $H_3$ hold, then problem (\ref{eq1}) admits five nontrivial smooth solutions
	\begin{eqnarray*}
		&&u_0,\hat{u}\in {\rm int}\,C_+,\hat{u}-u_0\in C_+\backslash\{0\},\\
		&&v_0,\hat{v}\in-{\rm int}\,C_+,v_0-\hat{v}\in C_+\backslash\{0\},\\
		&&y_0\in[v_0,u_0]\cap C^1_0(\overline{\Omega})\ \mbox{nodal}.
	\end{eqnarray*}
\end{theorem}

\begin{proof}
	Propositions \ref{prop10} and \ref{prop12} provide three nontrivial smooth solutions
	\begin{eqnarray*}
		&&u_0\in[0,w_+]\cap {\rm int}\,C_+\ \mbox{with}\ (w_+-u_0)(z)>0\ \mbox{for all}\ z\in\overline{\Omega},\\
		&&v_0\in[w_-,0]\cap(-{\rm int}\,C_+)\ \mbox{with}\ (u_0-w_-)(z)>0\ \mbox{for all}\ z\in\overline{\Omega},\\
		&&y_0\in[v_0,u_0]\cap C^1_0(\overline{\Omega})\ \mbox{nodal}.
	\end{eqnarray*}
	
	On account of Proposition \ref{prop11}, we may assume that $u_0$ and $v_0$ are extremal constant sign solutions (that is, $u_0=u_*$ and $v_0=v_*$).
	
	We consider the Carath\'eodory function $\gamma_+:\Omega\times\RR\rightarrow\RR$ defined by
	\begin{equation}\label{eq96}
		\gamma_+(z,x)=\left\{\begin{array}{ll}
			f(z,u_0(z))&\mbox{if}\ x\leq u_0(z)\\
			f(z,x)&\mbox{if}\ u_0(z)<x
		\end{array}\right.
	\end{equation}
	and set $\Gamma_+(z,x)=\int^x_0\gamma_+(z,s)ds$. we consider the $C^1$-functional $\sigma_+:W^{1,p}_0(\Omega)\rightarrow\RR$ defined by
	$$\sigma_+(u)=\frac{1}{p}||Du||^p_p+\frac{1}{2}||Du||^2_2-\int_{\Omega}\Gamma_+(z,u)dz\ \mbox{for all}\ u\in W^{1,p}_0(\Omega).$$
	
	Using (\ref{eq96}) we can easily show that
	\begin{equation}\label{eq97}
		K_{\sigma_+}\subseteq\left[u_0\right)=\{u\in W^{1,p}_0(\Omega):u_0(z)\leq u(z)\ \mbox{for almost all}\ z\in\Omega\}.
	\end{equation}
	
	Note that $u_0\in K_{\sigma_+}$. We may assume that
	\begin{equation}\label{eq98}
		K_{\sigma_+}\cap[u_0,w_+]=\{u_0\}.
	\end{equation}
	
	Otherwise, we already have a second positive solution $\hat{u}\geq u_0,\hat{u}\neq u_0,\hat{u}\in C^1_0(\overline{\Omega})$. Consider the following Carath\'eodory function
	\begin{equation}\label{eq99}
		\hat{\gamma}_+(z,x)=\left\{\begin{array}{ll}
			\gamma_+(z,x)&\mbox{if}\ x\leq w_+(z)\\
			\gamma_+(z,w_+(z))&\mbox{if}\ w_+(z)<x.
		\end{array}\right.
	\end{equation}
	
	We set $\hat{\Gamma}_+(z,x)=\int^x_0\hat{\gamma}_+(z,s)ds$ and consider the $C^1$-functional $\hat{\sigma}_+:W^{1,p}_0(\Omega)\rightarrow\RR$ defined by
	$$\hat{\sigma}_+(u)=\frac{1}{p}||Du||^p_p+\frac{1}{2}||Du||^2_2-\int_{\Omega}\hat{\Gamma}_+(z,u)dz\ \mbox{for all}\ u\in W^{1,p}_0(\Omega).$$
	
	From (\ref{eq99}) it is clear that $\hat{\sigma}_+$ is coercive. Also, it is sequentially weakly lower semicontinuous. So, we can find $\tilde{u}_0\in W^{1,p}_0(\Omega)$ such that
	\begin{equation}\label{eq100}
		\hat{\sigma}_+(\tilde{u}_0)=\inf[\hat{\sigma}_+(u):u\in W^{1,p}_0(\Omega)].
	\end{equation}
	
	Using (\ref{eq99}), we show that
	\begin{equation}\label{eq101}
		K_{\hat{\sigma}_+}\subseteq[u_0,w_+]\ (\mbox{see also (\ref{eq97})}).
	\end{equation}
	
	Then (\ref{eq98}), (\ref{eq100}), (\ref{eq101}) imply that
	\begin{equation}\label{eq102}
		\tilde{u}_0=u_0\in[0,w_+]\ (w_+-u_0)(z)>0\ \mbox{for all}\ z\in\overline{\Omega}.
	\end{equation}
	
	From (\ref{eq99}) we see that
	\begin{eqnarray*}
		&&\sigma_+|_{[0,w_+]}=\hat{\sigma}_+|_{[0,w_+]},\\
		&\Rightarrow&u_0\ \mbox{is a local}\ C^1_0(\overline{\Omega})-\mbox{minimizer of}\ \sigma_+\ (\mbox{see (\ref{eq102})}),\\
		&\Rightarrow&u_0\ \mbox{is a local}\ W^{1,p}_0(\Omega)-\mbox{minimizer of}\ \sigma_+\ (\mbox{see Proposition \ref{prop4}}).
	\end{eqnarray*}
	
	Because of (\ref{eq97}) we see that we may assume that $K_{\sigma_+}$ is finite or otherwise  we already have an infinity of positive, smooth (by the nonlinear regularity theory) solutions of (\ref{eq1}), all bigger than $u_0$. Hence, we can find  small $\rho\in(0,1)$ such that
	\begin{equation}\label{eq103}
		\sigma_+(u_0)<\inf[\sigma_+(u):||u-u_0||=\rho]=m^+_{\rho}.
	\end{equation}
	
	Reasoning as in the proof of Proposition \ref{prop5}, we can establish that
	\begin{equation}\label{eq104}
		\sigma_+\ \mbox{satisfies the C-condition}
	\end{equation}
	(in this case, due to (\ref{eq96}), for any Cerami sequence $\{u_n\}_{n\geq 1}\subseteq W^{1,p}_0(\Omega)$ we have automatically that $\{u^-_n\}_{n\geq 1}\subseteq W^{1,p}_0(\Omega)$ is bounded).
	
	Hypotheses $H_3(i),(ii)$ imply that we can find $\vartheta>\hat{\lambda}_m(p)$ and $c_{16}>0$ such that
	\begin{equation}\label{eq105}
		F(z,x)\leq\frac{\vartheta}{p}x^p+c_{16}\ \mbox{for almost all}\ z\in\Omega,\ \mbox{all}\ x\geq 0.
	\end{equation}
	
	Since $\hat{u}_1(p)\in {\rm int}\,C_+$, we can find $t\geq 1$ big such that $t\hat{u}_1(p)\geq u_0$ (see Proposition 2.1 of Marano and Papageorgiou \cite{20}). Then
	\begin{eqnarray}\label{eq106}
		 \sigma(t\hat{u}_1(p))&\leq&\frac{t^p}{p}\hat{\lambda}_1(p)+\frac{t^2}{2}||D\hat{u}_1(p)||^2_2-\frac{t^p}{p}\vartheta+c_{17}\ \mbox{for some}\ c_{17}>0\nonumber\\
		&&(\mbox{see (\ref{eq105}) and recall that}\ ||\hat{u}_1(p)||_p=1)\nonumber\\
		&=&\frac{t^p}{p}[\hat{\lambda}_1(p)-\vartheta]+\frac{t^2}{2}||D\hat{u}_1(p)||^2_2+c_{17}.
	\end{eqnarray}
	
	Since $\vartheta>\hat{\lambda}_1(p)$ and $p>2$ from (\ref{eq106}) it follows that
	\begin{equation}\label{eq107}
		\sigma(t\hat{u}_1(p))\rightarrow-\infty\ \mbox{as}\ t\rightarrow+\infty\,.
	\end{equation}
	
	Then (\ref{eq103}), (\ref{eq104}), (\ref{eq107}) permit the use of Theorem \ref{th1} (the mountain pass theorem). So, we can find $\hat{u}\in W^{1,p}_0(\Omega)$ such that
	\begin{equation}\label{eq108}
		\hat{u}\in K_{\sigma_+}\ \mbox{and}\ m^+_{\rho}\leq\sigma_+(\hat{u}).
	\end{equation}
	
	From (\ref{eq96}), (\ref{eq97}), (\ref{eq103}) and (\ref{eq108}) it follows that
	$$u_0\leq\hat{u},\ \hat{u}\neq u_0\ \mbox{and}\ \hat{u}\in {\rm int}\,C_+\ \mbox{is a solution of (\ref{eq1})}.$$
	
	Similarly, working with $v_0\in[w_-,0]\cap (-{\rm int}\,C_+)$ on the negative semiaxis as above, we produce $\hat{v}\in-{\rm int}\,C_+$, $\hat{v}\leq v_0$, $\hat{v}\neq v_0$, a second negative solution for problem (\ref{eq1}).
\end{proof}

\medskip
{\bf Acknowledgments.}  The authors were supported in part by the Slovenian Research Agency program P1-0292 and projects N1-0064, J1-8131, J1-7025, and J1-6721. V. R\u{a}dulescu was also supported by a grant of the Romanian National
Authority for Scientific Research and Innovation, CNCS-UEFISCDI, project number PN-III-P4-ID-PCE-2016-0130.

\end{document}